\documentclass[12pt,twoside]{amsart}
\usepackage{amssymb}
\usepackage{a4wide}
\usepackage[latin1]{inputenc}
\usepackage[T1]{fontenc}
\usepackage{times}

\def\hpq0{h^{p,q}_{\leq 0}}
\def\Hpq0{\H_{\leq 0}^{p,q}}

\def\dt{\partial^{\phi_t}}
\def\dbar{\bar\partial}
\def\ddbar{\partial\dbar}

\def\R{{\mathbb R}}

\def\C{{\mathbb C}}
\def\Po{{\mathbb P}}

\def\G{{\mathcal G}}

\def\F{{\mathcal F}}

\def\dof{\dot{\phi_t}}
\def\dofnu{\dot{\phi_t^\nu}}

\def\V{{\mathcal V}}

\def\H{{\mathcal H}}
\def\E{{\mathcal E}}

\def\Re{{\rm Re\,  }}
\def\Im{{\rm Im\,  }}

\def\be{\begin{equation}}
\def\ee{\end{equation}}

\newtheorem{thm}{Theorem}[section]
\newtheorem{lma}[thm]{Lemma}

\newtheorem{prop}[thm]{Proposition}

\theoremstyle{definition}

\newtheorem{df}{Definition}

\theoremstyle{remark}

\newtheorem{preremark}{Remark}
\newtheorem{preex}{Example}

\numberwithin{equation}{section}

\begin{document}

\title[]
{A Brunn-Minkowski type inequality for Fano manifolds and some uniqueness theorems in K\"ahler geometry.}

\address{B Berndtsson :Department of Mathematics\\Chalmers University
  of Technology 
  \\S-412 96
  G\"OTEBORG\\SWEDEN,\\} 

\email{ bob@chalmers.se}

\author[]{ Bo Berndtsson}

\begin{abstract}
For $\phi$ a metric on the anticanonical bundle, $-K_X$, of a Fano manifold $X$
we consider the volume of $X$
$$
\int_X e^{-\phi}.
$$
In earlier papers we have proved  that the logarithm of the volume is concave along  geodesics
in the space of positively curved metrics on $-K_X$. Our main result here is  that the
concavity is strict unless the geodesic comes from the flow of a
holomorphic vector field on $X$, even with very low regularity assumptions on the geodesic. As a consequence we get a simplified
proof of the Bando-Mabuchi uniqueness theorem for K\"ahler - Einstein
metrics. A generalization of this theorem to 'twisted'
K\"ahler-Einstein metrics and  some classes of manifolds that
satisfy weaker hypotheses than being Fano is also given. We moreover discuss a generalization of the main result to other bundles than $-K_X$, and finally use the same method to give a new proof of the theorem of Tian and Zhu on uniqueness of K\"ahler-Ricci solitons. 
\end{abstract}

\bigskip

\maketitle

\section{Introduction}

Let $X$ be an $n$-dimensional  projective  manifold with
seminegative canonical 
bundle and let $\Omega$ be a domain in the complex plane. We consider
curves $t\rightarrow \phi_t$, with $t$ in $\Omega$,  of metrics on $-K_X$  that have
plurisubharmonic 
variation so that $i\ddbar_{t, X}\phi\geq 0$ ( see section 2 for
notational conventions). Then $\phi$ solves the
homogenous Monge-Amp\`ere equation if  
\be
(i\ddbar_{t,X}\phi)^{n+1}=0.
\ee
Such curves are called (generalized) geodesics, see \cite{1Mabuchi} for the origins of this. 

By a
fundamental theorem of Chen,  \cite{Chen}, we can for any given $\phi_0$ defined
on the boundary of $\Omega$, smooth with nonnegative curvature on $X$
for $t$ fixed on $\partial\Omega$, find a solution of (1.1)
with $\phi_0$ as boundary values. This solution does in general not
need to be smooth (see \cite{1Donaldson},\cite{Lempert-Vivas}, \cite{Darvas}), but Chen's theorem asserts
that we can find a solution that has all mixed complex derivatives
bounded, i e $\ddbar_{t, X}\phi$ is bounded on $X\times \Omega$. The
solution equals the 
supremum (or maximum) of all subsolutions, i e all metrics with
semipositive curvature that are dominated by $\phi_0$ on the
boundary. Chen's proof is based on some of  the methods from Yau's proof of the
Calabi conjecture, so it is not so easy, but it is worth pointing out
that the existence of 
a generalized solution that is only bounded is much easier, see
section 2. 

On the
other hand, if we do assume that $\phi$ is smooth and $i\ddbar_X\phi>0$
on $X$ for any 
$t$ fixed, then
$$
(i\ddbar_{t,X}\phi)^{n+1}=n c(\phi)(i\ddbar_X\phi)^n\wedge i dt\wedge d\bar t
$$
with 
$$
c(\phi)=\frac{\partial^2\phi}{\partial t\partial\bar
  t}-|\dbar\frac{\partial\phi}{\partial t}|^2_{i\ddbar_X\phi},
$$
where the norm in the last term is the norm with respect to the K\"ahler metric
$i\ddbar_X\phi$. Thus equation 1.1 is then equivalent to $c(\phi)=0$. 

The case when $\Omega=\{t; 0<\Re t<1\}$ is a strip  is of
particular interest. If the boundary data are  independent of $\Im
t$ then so is the  solution to 1.1. A famous
observation of Semmes, \cite{Semmes} and Donaldson, \cite{Donaldson}
is that the equation $c(\phi)=0$ then is the equation for a geodesic
in the space of K\"ahler potentials. Chen's theorem then {\it almost}
implies that 
any two points in the space of K\"ahler potentials can be joined by a 
geodesic, the proviso being that we might not be able to keep
smoothness or strict positivity along all of the curve. This problem
causes some difficulties in applications, one of which we will
address in this paper.

The next theorem is a direct consequence of the results in
\cite{1Berndtsson}, \cite{3Berndtsson}.
\begin{thm}Assume that $-K_X\geq 0$ in the sense that it has a smooth metric of semipositive curvature. Let
let $\phi_t$ be a  curve of (possibly singular)   metrics on $-K_X$
such that
$$
i\ddbar_{t,X}\phi\geq 0
$$
in the sense of currents. Then 
$$
\F(t):=-\log\int_X e^{-\phi_t}.
$$
 is subharmonic in $\Omega$. In particular, if $\phi_t$ does not
 depend on the imaginary part of $t$, $\F$ is convex. 
\end{thm}

Here we interpret the integral over $X$ in the following way. For any
choice of local coordinates $z^j$ in some covering of $X$ by
coordinate neighbourhoods $U_j$, the metric $\phi_t$ is
represented by a local function $\phi_t^j$. The volume form
$$
c_n e^{-\phi^j_t}  dz^j\wedge d\bar z^j,
$$
where $c_n=i^{n^2}$ is a unimodular constant chosen to make the form
positive, is independent of the choice of local coordinates. We denote this
volume form by $e^{-\phi_t}$, see section 2. 

The results in \cite{1Berndtsson} and \cite{3Berndtsson} 
 deal with more
general line bundles $L$  over $X$ and also more general fibrations than $X\times\Omega$, see section 3. A special case is the trivial vector
bundle $E$ over $\Omega$ with fiber $H^0(X, K_X+L)$ with the
$L^2$-metric
$$
\|u\|^2_t=\int_X |u|^2 e^{-\phi_t},
$$
see section 2. The main result is then a formula for the curvature of
$E$ with the $L^2$-metric. In this paper we study primarily  the simplest special
case, $L=-K_X$. Then $K_X+L$ is trivial so $E$ is a line bundle and
Theorem 1.1 says that this line bundle has nonnegative curvature. In section 9 we shall  be able to extend part of the results we now describe to more general line bundles than $-K_X$. In case $-K_X>0$, so that $X$ is Fano, the result is a simple consequence of H\"ormander's $L^2$-estimates, see \cite{2Berndtsson} for a very short proof in this case. 

Theorem 1.1 is formally analogous to the Brunn-Minkowski inequality
for the volumes of convex sets, and even more to its functional version,
Prekopa's theorem, \cite{Prekopa}. Prekopa's theorem states that if
$\phi$ is a convex function on $\R^{n+1}$, 
then
$$
f(t):=-\log\int_{\R^n}e^{-\phi_t}
$$
is convex. The complex counterpart of this is that we consider a
complex manifold $X$ with a family of volume forms $\mu_t$. In local
coordinates $z^j$  the volume form can be written as above
$\mu_t=c_n e^{-\phi^j_t} dz^j\wedge\bar dz^j$ , and if $\mu_t$ is
globally well defined $\phi^j_t$ are then the local representatives of
a metric, $\phi_t$, on $-K_X$. Convexity in Prekopa's theorem then
corresponds to 
positive, or at least semipositive, curvature of $\phi_t$, so $X$ must
be Fano, or its canonical bundle must  at least have seminegative
curvature  (in some sense:
$-K_X$ pseudoeffective would be the minimal requirement). The
assumption in Prekopa's theorem that the weight is convex with respect
to $x$ and $t$ together then corresponds to the assumptions in Theorem
1.1.

If $K$ is a compact convex set in $\R^{n+1}$ we can take $\phi$ to
be equal to 0 in $K$ and $+\infty$ outside of $K$. Prekopa's theorem
then implies the Brunn-Minkowski theorem, saying that the logarithm of
the volumes of $n$-dimensional slices, $K_t$ of convex sets are
concave; concretely 
\be
|K_{(t+s)/2}|^2\leq |K_t||K_s|
\ee

The Brunn-Minkowski theorem has an important addendum which describes
the case of equality : If equality holds in (1.2)
then all the slices $K_t$ and $K_s$  are translates of each other
$$
K_t=K_s + (t-s)\mathbf v
$$
where $\mathbf v$ is some  vector in $\R^n$.
A little bit
artificially we can formulate this as saying that we move from one
slice to another via the flow of a constant vector field. 

\begin{preremark}
It follows that from (1.2) and the natural homogenity properties of
Lebesgue measure that $|K_t|^{1/n}$, is also concave. This ('additive
version')  is perhaps
the most common formulation of the Brunn-Minkowski inequalities, but
the logarithmic (or multiplicative) version above works better for
weighted volumes and in the complex setting.
 For the additive version
conditions for equality are more liberal; then $K_t$ may change not
only by translation but also by dilation 
(see \cite{Gardner}), but equality in the multiplicative case excludes
dilation. 
\end{preremark}

A natural
question is then if one can draw a similar conclusion in the complex
setting described above. In \cite{2Berndtsson} we proved that this is
indeed so if $\phi$ is known to be smooth and strictly
plurisubharmonic on $X$ for $t$ fixed. The main result of this paper
is the extension of this to less regular situations. We keep the same
assumptions as in Theorem 1.1. 
\begin{thm}Assume that $H^{0,1}(X)=0$, and that the curve of metrics
 $\phi_t$ is 
  independent of the imaginary part of $t$.  Assume moreover that the
  metrics $\phi_t$ are uniformly bounded in the sense that for some
  smooth metric on $-K_X$, $\psi$,
$$
|\phi_t-\psi|\leq C.
$$
Then, if the function $\F$ in Theorem 1.1 is affine in $\Omega$, 
there is a  holomorphic vector field $V$ on
$X$ with flow $F_t$ such 
that
$$
F_t^*(\ddbar\phi_t) =\ddbar\phi_0.
$$
\end{thm}
The assumption that $H^{0,1}(X)=0$ enters into the proof at several places, but I do not know if it is necessary for the theorem to hold. Notice however that it is automatically satisfied if $X$ is Fano. Then $-K_X>0$ so $H^{0,1}(X)=H^{n,1}(X, -K_X)=0$ by Kodaira vanishing. More generally, if $-K_X$ is supposed to be 'big', $H^{n,1}(X, -K_X)$ also vanishes by the Demailly-Nadel vanishing theorem.

There should also be a version of the theorem  without the assumption that
$\phi_t$ be independent of the imaginary part of $t$, and then
assuming that $\F$ be harmonic instead of affine.  The proof
then seems to require more regularity assumptions. For simplicity we
therefore treat only the case when $\phi_t$ is independent of $\Im t$,
which anyway seems to be the most useful in applications. 
\bigskip

\noindent This theorem is useful in view of the discussion above on the possible lack of
regularity of geodesics. As we shall
see in section 2 the existence of a generalized geodesic satisfying
the boundedness assumption in Theorem 1.2 is almost trivial.  One
motivation for the theorem is to give a new proof of the Bando-Mabuchi
uniqueness theorem for K\"ahler-Einstein metrics on Fano
manifolds. Recall that a metric $\omega_\psi=i\ddbar\psi$, with
$\psi$ a metric on $-K_X$ solves the  K\"ahler-Einstein equation if 
$$
\text{Ric}(\omega_\psi)=\omega_\psi
$$
or equivalently if for some positive $a$
\be
e^{-\psi}=a(i\ddbar\psi)^n,
\ee
where we use the convention above to interpret $e^{-\psi}$ as a volume
form. By a celebrated theorem of Bando and Mabuchi (see section 5),  any two
K\"ahler-Einstein metrics $i\ddbar\phi_0$ and 
$i\ddbar\phi_1$ are related via the time-one flow of a holomorphic vector
field. In section 5 we shall give a proof of this fact by
joining $\phi_0$ and $\phi_1$ by a geodesic and applying Theorem
1.2. This proof also shows that the uniqueness theorem of Bando-Mabuchi holds also for solutions of (1.3) that are only assumed to be bounded. The original proof of Bando and Mabuchi used monotonicity properties of the Mabuchi K-energy (\cite{2Mabuchi})  along curves obtained from solving a continuous family of Monge-Ampere equations, and thus seems to require higher regularity. Below we will also consider 'twisted' K\"ahler-Einstein equations, whose solutions are never smooth, and then this difference between the proofs is perhaps more important.

It should be noted that a similar proof of the  Bando-Mabuchi
theorem has already been given by
Berman, \cite{Berman}. The difference between his proof and ours is
that he uses the weaker version of Theorem 1.2 from
\cite{2Berndtsson}. He then needs to prove that the geodesic joining
two K\"ahler-Einstein metrics is in fact smooth, which we do not
need, and we also avoid the use  of 
Chen's theorem since we only need the existence of a bounded
geodesic.

 A minimal assumption in Theorem 1.2 would be that $e^{-\phi_t}$ be
 integrable, instead of bounded. I do not know if the theorem holds
 in this generality, but in section 6 we will consider an
 intermediate situation where $\phi_t=\tau_t +\psi$, with $\tau_t$
 bounded and $\psi$ such that $e^{-\psi}$ is integrable, so that the
 singularities don't change with $t$. Under various positivity
 assumptions we are then able to prove a version of Theorem 1.2.
 
Apart from making the problem technically simpler, this extra
assumption that  $\phi_t=\tau_t +\psi$ also introduces an additional
structure, which seems interesting in itself. 
In section 7 we use it to  give a generalization of the Bando-Mabuchi
theorem to certain 'twisted' K\"ahler-Einstein equations, 
\be
\text{Ric}(\omega)=\omega +\theta
\ee
considered
in \cite{Szekelyhidi},\cite{2Berman} and  \cite{2Donaldson}. Here
$\theta$ is a fixed 
positive $(1,1)$-current, that may  e g be the current of integration
on a klt divisor.
The conclusion of our theorem is  that in (1.4) we have uniqueness modulo the time one flow of a vector field that fixes $\theta$. We shall also see, in section 8,  that in many cases, this means that we in fact have absolute uniqueness.

After this, in section 9,  we briefly discuss a variant of Theorem 1.2 for more general line bundles, $L$, than $-K_X$. We then replace the functional 
$$
\F(t)=-\log\int e^{-\phi_t},
$$
by a variant,  introduced in \cite{2Berndtsson},  of Donaldson's $L$-functional, \cite{3Donaldson}. Finally, in section 10, following a suggestion of Yanir Rubinstein, we show how Theorem 1.2 also implies a theorem of Tian and Zhu, \cite{Tian-Zhu}, on uniqueness for K\"ahler-Ricci solitons. This has also been noted independently by W He, \cite{He-man}.

Another paper that is very much related to this one is \cite{Berman
  et al}, by Berman -Boucksom-Guedj-Zeriahi. There is introduced a
variational approach to Monge-Ampere equations and K\"ahler-Einstein
equations in a nonsmooth setting and a uniqueness theorem a la
Bando-Mabuchi is proved in the absence of holomorphic vector fields, using continuous geodesics. After the first version of this paper was written, the results have also been generalized to some singular varieties in \cite{2Berman et al}. I would like to thank all of these authors   for
helpful discussions, and Robert Berman in particular for
proposing the generalized Bando-Mabuchi theorem in section 7. Finally I am very grateful to two referees for valuable comments, in particular for a suggestion how to prove that the vector field in Theorem 1.2 is time independent. 

\section{Preliminaries}

\subsection{Notation}
Let $L$ be a line bundle over a complex manifold $X$, and let $U_j$ be a
covering of the manifold by open sets over which $L$ is locally
trivial. A section of $L$ is then represented by a collection of
complex valued functions $s_j$ on $U_j$ that are related by the
transition functions of the bundle, $s_j=g_{j k} s_k$. A metric on $L$
is given by a collection of realvalued functions $\phi^j$ on $U_j$,
related so that
$$
|s_j|^2 e^{-\phi^j}=:|s|^2 e^{-\phi}=:|s|^2_\phi
$$
is globally well defined. We will write $\phi$ for the collection
$\phi^j$, and refer to $\phi$ as the metric on $L$, although it might
be more appropriate to call $e^{-\phi}$ the metric. (Some authors call
$\phi$ the 'weight' of the metric.) We say that $L$ is positive, $L>0$,  if $\phi$ can be chosen smooth with curvature $i\ddbar\phi$ strictly positive, and that $L$ is semipositive, $L\geq 0$,  if it has a smooth metric of semipositive curvature.  

A metric $\phi$ on $L$ induces an $L^2$-metric on the adjoint bundle
$K_X+L$. A section $u$ of $K_X+L$ can be written locally as
$$
u= dz\otimes s
$$
where $dz=dz_1\wedge ...dz_n$ for some choice of local coordinates
and $s$ is a section of $L$. We let
$$
|u|^2 e^{-\phi}:= c_n dz\wedge d\bar z |s|_\phi^2;
$$
it is a volume form on $X$. The $L^2$-norm of $u$ is
$$
\|u\|^2:=\int_X |u|^2 e^{-\phi}.
$$
Note that the $L^2$ norm depends only on the metric $\phi$ on $L$ and
does not involve any choice of metric on the manifold $X$. 

In this paper we will be mainly interested in the case when $L=-K_X$
is the anticanonical bundle. Then the adjoint bundle $K_X+L$ is
trivial and is canonically isomorphic to $X\times \C$ if we have
chosen an isomorphism between $L$ and $-K_X$. This bundle then has a
canonical trivialising section, $u$ identically equal to 1. With the
notation above
$$
\|1\|^2 =\int_X |1|^2 e^{-\phi}=\int_X e^{-\phi}.
$$
This means explicitly that we interpret the volume form
$e^{-\phi}$ as 
$$
dz^j\wedge d\bar z^j e^{-\phi_j}
$$
where $e^{-\phi^j}= |(dz^j)^{-1}|_\phi^2$ is the local representative of
the metric for the frame determined by the local coordinates. Notice
that this is consistent with the conventions indicated in the
introduction. 

\subsection{Bounded geodesics}
We now consider curves $t\rightarrow \phi_t$ of metrics on the line bundle
$L$. Here $t$ is a complex parameter but we shall (almost) only look
at curves that do not depend on the imaginary part of $t$. We say that
$\phi_t$ is a subgeodesic if $\phi_t$ is upper semicontinuous and
$i\ddbar_{t, X}\phi_t\geq 0$,  so that local
representatives are plurisubharmonic with respect to $t$ and $X$
jointly. We say that $\phi_t$ is bounded if
$$
|\phi_t-\psi|\leq C
$$
for some constant $C$ and some (hence any) smooth metric on $L$. For
bounded geodesics the complex Monge-Ampere operator is well defined
and we say that $\phi_t$ is a (generalized) geodesic if
$$
(i\ddbar_{t, X} \phi_t)^{n+1}=0.
$$

\bigskip

\noindent Let $\phi_0$ and $\phi_1$ be two bounded metrics on $L$ over
$X$ satisfying
$i\ddbar\phi_{0,1}\geq 0$.  We claim that there is a bounded geodesic
$\phi_t$ defined for the real part of $t$ between 0 and 1, such that
$$
\lim_{t\rightarrow 0,1} \phi_t =\phi_{0,1}
$$
uniformly on $X$. The curve $\phi_t$ is defined by
\be
\phi_t=\sup \{\psi_t\}
\ee
where the supremum is taken over all plurisubharmonic $\psi_t$ with 
$$
 \lim_{t\rightarrow 0,1} \psi_t \leq \phi_{0,1}.
$$
To prove that $\phi_t$ defined in this way has the desired properties
we first construct a barrier 
$$
\chi_t =\text{max} (\phi_0 -A\Re t, \phi_1 +A(\Re t -1)).
$$
Clearly $\chi$ is plurisubharmonic and has the right boundary values
if $A$ is sufficiently large. Therefore the supremum in (2.1) is the
same if we restrict it to $\psi$ that are larger than $\chi$. For such
$\psi$ the onesided derivative at 0 is larger than $-A$ and the
onesided derivative at 1 is smaller than  $A$. Since we may moreover
assume that $\psi$ is independent of the imaginary part of $t$, $\psi$
is convex in $t$ so the
derivative with respect to $t$ increases, and must therefore lie
between $-A$ and $A$. Hence $\phi_t$ satisfies
$$
\phi_0 -A\Re t\leq \phi_t\leq \phi_0 +A\Re t
$$
and a similar estimate at 1. Thus $\phi_t$ has the right boundary
values uniformly. In addition, the upper semicontinuous regularization
$\phi_t^*$ 
of $\phi_t$ must satisfy the same estimate. Since $\phi_t^*$ is
plurisubharmonic it belongs to the class of competitors for $\phi_t$
and must therefore coincide  with $\phi_t$, so $\phi_t$ is
plurisubharmonic. That finally $\phi_t$ solves the homogenuous
Monge-Ampere equation follows from the fact that it is maximal with
given boundary values, see e g \cite{Guedj-Zeriahi}, Thm 2.20. 

Notice that as a byproduct of the proof we have seen that the geodesic
joining two bounded metrics is uniformly Lipschitz in $t$. This fact
will be very useful later on. 
\subsection{Approximation of metrics and subgeodesics}

In the proofs we will need to approximate our metrics that are only
bounded, and sometimes not even bounded, by smooth metrics. Since we
do not want to lose too much of the positivity of curvature
this causes some complications. An extensive treatment of
these matters can be found in \cite{Demailly}. Here we will need only
the simplest part of this theory and we also refer to
\cite{Blocki-Kolodziej} for an elementary proof. We collect the approximation  results that we need in a proposition.
\begin{prop} Let $M$ be a complex manifold with a positive hermitean form $\omega$, and let $L$ be a complex line bundle over $M$. Let $\phi$ be a bounded metric on $L$ such that $i\ddbar\phi\geq 0$. Let $M'$ be a relatively compact domain in $M$ (which could be $M$ itself if $M$ is compact). Then there is a strictly decreasing sequence $\phi_j$ of smooth metrics on $L$ over $M'$ with limit $\phi$, such that
$$
i\ddbar\phi_j\geq -\epsilon_j\omega,
$$
where $\epsilon_j>0$ tends to zero. Moreover:

(1) If $L\geq 0$ this result holds without the assumption that $\phi$ be bounded 

and

(2) If $L>0$, $\phi_j$ can be chosen so that $i\ddbar\phi_j>0$, and the result holds without the assumption that $\phi$ be bounded.

\end{prop}
\begin{proof} This is basically the main result in \cite{Blocki-Kolodziej}, and for the convenience of the reader we translate to the language of $\gamma$-plurisubharmonic functions used in that paper. Let $\psi$ be a smooth metric on $L$ and let $\gamma:=i\ddbar\psi$. To any metric $\phi$ on $L$, we associate the function $\varphi:=\phi-\psi$. The condition $i\ddbar\phi\geq 0$ then says that $\varphi$ is $\gamma$-plurisubharmonic, i e that
$$
i\ddbar\varphi\geq -\gamma.
$$
Similarily, $i\ddbar\phi>0$ means that $i\ddbar\varphi >- \gamma$, and $i\ddbar\phi\geq- \epsilon\omega$ means that $\phi$ is $(\gamma+\epsilon\omega)$-plurisubharmonic. The first statement of the proposition is (a special case of)  Theorem 2 in \cite{Blocki-Kolodziej}. For  statement (2) concerning positive bundles, we can assume that $\gamma>0$. Choose $\phi_j$ as in the first part, and let $\varphi_j:=\phi_j-\psi$. Since $\phi_j$ are smooth and decrease, we may assume these functions are negative. Then, if $\delta_j$ decrease to zero, $(1-\delta_j)\varphi_j$ decrease and $i\ddbar (1-\delta_j)\varphi_j\geq -(1-\delta_j)(\gamma +\epsilon_j\omega)>-\gamma$, if $\delta_j$ goes to zero sufficiently slowly. Thus $\phi$ can be approximated with a sequence of metrics of strictly positive curvature. If $\phi$ is not bounded, we apply this argument to $\varphi^A:=\max(\varphi, -A)$, if $A>0$. For each $A$ we get a sequence, $\varphi_j^A$  of strictly $\gamma$-plurisubharmonic functions that decrease to $\varphi^A$. Then take a sequence $A_\nu$ that increases to infinity and let
$$
\varphi_\nu:=\varphi_{j_\nu}^{A_\nu},
$$
where $j_\nu$ is chosen inductively so that 
$$
\varphi_{j_{\nu+1}}^{A_{\nu+1}}<\varphi_{j_\nu}^{A_\nu}.
$$
This is possible by Dini's lemma since $\varphi_j^{A_{\nu+1}}$ is a decreasing sequence of continuous functions whose limit, $\varphi^{A_{\nu+1}}$ is strictly smaller than the right hand side. This argument also proves (1). 
\end{proof}
Besides using Proposition 2.1 to approximate metrics on a line bundle over $X$, we can also apply it to the manifold $S\times X$, $S=\{t;0<\Re t<1\}$, to approximate (sub)geodesics over any relatively compact subdomain of $S$. In case the (sub)geodesic depends only on $\Re t$ we can then obtain smooth approximants that also depend only on $\Re t$. To see this, we  replace $S$ by an annulus by a conformal change of coordinates in $t$, and take averages of $\phi_j$ over the circle.

At one point we also wish to treat a bundle that is not even
semipositive, but only effective. It then has a global holomorphic
section, $s$, and the singular metric we are interested in is
$\log|s|^2$, or some positive multiple of it. We then let $\psi$ be
any smooth metric on the bundle and approximate by
$$
\phi^\nu:=\log( |s|^2 + \nu^{-1}e^{\psi}).
$$
Explicit computation shows that $i\ddbar\phi^\nu\geq -C\omega$ where $C$
is some fixed constant. Moreover, outside any fixed neighbourhood of
the zerodivisor of $s$,
$$
i\ddbar\phi^\nu\geq -\epsilon_\nu\omega
$$
with $\epsilon_\nu$ tending to zero. This weak approximation will be
enough for our 
purposes.

Let us finally note that we know from the barrier construction in the previous subsection that a bounded geodesic $\phi_t$ has uniformly bounded $t$-derivative, $\dof$. A similar argument shows that
 an approximating sequence $\phi^\nu$, decreasing to a bounded geodesic $\phi$, also can be chosen so that it has uniformly bounded $t$-derivative. For this it is enough to replace $\phi^\nu$ by
$$
\max(\phi^\nu_t, \max(\phi_0^\nu-A\Re t, \phi_1^\nu+A(\Re t-1))).
$$
This function still decreases to $\phi$ and has derivative bounded between $A$ and $-A$. It is not smooth because of the max construction, but we can replace  the maximum by a smoothed out version of max. The upshot of this is that we will (see lemma 4.1) also get dominated convergence almost everywhere for the time derivatives of the approximating sequence.

\subsection{Monge-Ampere energy} In this subsection we collect some basic properties of the Monge-Ampere energy. These facts are  well known at least in the smooth case; our purpose here is to check that they still hold for bounded curves, and we follow the arguments in  \cite{Berman}. Let $\phi_0$ and $\phi_1$ be two bounded metrics on a line bundle $L$, satisfying $i\ddbar\phi_j\geq 0$. Then their relative Monge-Ampere energy
\be
\E(\phi_1,\phi_0):=(1/n)\int_X (\phi_1-\phi_0)\sum_0^n (i\ddbar\phi_1)^k\wedge(i\ddbar\phi_0)^{n-k}
\ee
is well defined by basic pluripotential theory. (We will change the normalization later by dividing by the volume of $L$.) It has the property that if $\phi_t$ depends smoothly on $t$, then 
$$
(d/dt)\E(\phi_t,\phi_0)=\int_X \dof (i\ddbar\phi_t)^n,
$$
and $\E(\phi_0,\phi_0)=0$; these properties are sometimes taken as an alternative definition of $\E$. We could also write, if $\phi_t$ is just a bounded subgeodesic,
$$
\E(\phi_t,\phi_0)= p_*((\phi_t-\phi_0))\sum_0^n (i\ddbar_{t,X}\phi_t)^k\wedge(i\ddbar\phi_0)^{n-k}),
$$
where $p$ is the natural projection from $X\times \Omega$ to $\Omega$, and $p_*$ is the pushforward of a current. Since the pushforward commutes with differentiation, the last formula shows that
$$
i\ddbar_t\E(\phi_t,\phi_0)=(1/n)p_*((i\ddbar_{t,X}\phi_t)^{n+1}-(i\ddbar\phi_0)^{n+1})= (1/n)p_*((i\ddbar_{t,X}\phi_t)^{n+1}).
$$
Using the definition of $c(\phi)$ from the introduction we can also write this as 
$$
i\ddbar_t\E(\phi_t,\phi_0)= \int_X c(\phi_t)(i\ddbar\phi_t)^n idt\wedge d\bar t.
$$
At any rate we see that $\E$ is convex along bounded subgeodesics and affine along bounded geodesics. It also follows (most easily from the last formula) that on an affine line $\phi_t=\phi_0 +t(\phi_1-\phi_0)$, $\E$ is {\it concave}, with derivative
$$
(d/dt)_{t=0}\E(\phi_t,\phi_0)=\int_X (\phi_1-\phi_0)(i\ddbar\phi_0)^n
$$
(use (2.2)). The concavity shows that
\be
\E(\phi_1,\phi_0)\leq \int_X (\phi_1-\phi_0)(i\ddbar\phi_0)^n.
\ee
If we replace $\phi_1$ by $\phi_t$ in (2.3), with $\phi_t$ a bounded subgeodesic we see by monotone convergence that the derivative of $\E$ from the right satisfies
\be
(d/dt)_{t=0,+}\E(\phi_t,\phi_0)\leq \int_X (\dot{\phi_0})_+(i\ddbar\phi_0)^n.
\ee
Similarily, the derivative at $t=1$ from the left satisfies
\be
(d/dt)_{t=1,-}\E(\phi_t,\phi_0)\geq \int_X (\dot{\phi_1})_-(i\ddbar\phi_1)^n.
\ee
We will have use for these formulas in section 5. 
\section{The smooth case}

In this section we let $L$ be a holomorphic line bundle over $X$ and
$\Omega$ be a smoothly bounded open set in $\C$. Fix once and for all one K\"ahler form on $X$, $\omega$. We consider
the trivial  vector bundle $E$ 
over $\Omega$ with fiber $H^0(X, K_X+L)$. In this section we let throughout  $\phi_t$ be a smooth
curve of 
metrics on $L$, with $t$ a complex parameter.  For any fixed $t$,
$\phi_t$ 
 induces an
$L^2$-norm on $H^0(X, K_X+L)$ as described in the previous section
$$
\|u\|^2_t=\int_X |u|^2 e^{-\phi_t},
$$
and as $t$ varies we get an hermitian metric on the vector bundle
$E$. 

We now recall a formula for the curvature of $E$ with this metric
from \cite{1Berndtsson},\cite{3Berndtsson}. Let for each $t$ in $\Omega$ 
$$
\dt= e^{\phi_t}\partial e^{-\phi_t}=\partial-\partial \phi_t\wedge.
$$
We let this operator act on $L$-valued forms, $v$, of bidegree $(n-1,0)$, and we interpret it locally in terms of some local trivialization. It can be easily checked that it is globally well defined. 
 
Let $v$ be an $L$-valued $(n-1,0)$-form and write
$\alpha=v\wedge \omega$, where $\omega$ is the fixed K\"ahler form on $X$.  Then (modulo a sign)
$$
\dt v=\dbar^*_{\phi_t}\alpha,
$$
the adjoint of the $\dbar$-operator for the metric $\phi_t$. In
particular this  shows again that the operator $\dt$ is well defined on
$L$-valued forms. 

This also
means that for any $t$ we can solve the equation
$$
\dt v=\eta,
$$
if $\eta$ is an $L$-valued $(n,0)$-form that is orthogonal to the
space of holomorphic $L$-valued forms (see remark 2 below). Moreover
by choosing $\alpha=v\wedge\omega$ 
orthogonal to the kernel of $\dbar^*_{\phi_t}$ we can assume that
$\alpha$ is $\dbar$-closed, so that $\dbar v\wedge \omega=0$.( Hence,
with this choice, $\dbar v$ is a primitive form.)
If, as we assume from now, the
cohomology $H^{n, 1}(X,L)=0$, the $\dbar$-operator is surjective on
$\dbar$-closed forms, so
the adjoint is injective, and $v$ is uniquely determined by $\eta$. 

\begin{preremark}
 The reason we can always solve this equation for $t$ and
$\phi$ 
fixed is that the $\dbar$-operator from $L$-valued $(n,0)$-forms to
$(n,1)$-forms  on $X$ has closed range. This implies that the adjoint
operator
$\dbar^*_{\phi_t}$ also has closed range and that its range is equal
to the orthogonal complement of the kernel of $\dbar$. Moreover, that
$\dbar$ has closed range means precisely that for any $(n,1)$-form in
the range of $\dbar$ we can solve the equation $\dbar f=\alpha$ with
an estimate 
$$
\|f\|\leq C\|\alpha\|
$$
and it follows from functional analysis that we then can solve $\dt
v=\eta$ with the bound
$$
\| v\|\leq C\|\eta\|
$$
where $C$ is {\it the same} constant. We apply these general facts to the norms $\|\cdot\|=\|\cdot\|_{\omega,\phi_t}$ defined by our fixed K\"ahler form $\omega$ and metrics $\phi_t$. In case all metrics $\phi_t$ are
of equivalent size, so that $|\phi_t-\phi_{t_0}|\leq A$ it follows
that we can solve $\dt v=\eta$ with an $L^2$-estimate independent of
$t$. This observation is of crucial importance in the sequel. 
\end{preremark}

Let $u_t$
be a holomorphic section of the bundle $E$ and let 
$$
\dof:=\frac{\partial\phi}{\partial t}.
$$

\bigskip

For each $t$ we now solve
\be
\dt v_t=\pi_\perp(\dof u_t),
\ee
where $\pi_\perp$ is the orthogonal projection on the orthogonal
complement of the space of holomorphic forms, with respect to the
$L^2$-norm  $\|\cdot\|_t^2$. With this choice of $v_t$ we obtain the
following formula for the curvature of $E$, see \cite{1Berndtsson},
\cite{3Berndtsson}. In the formula, $p$ stands for the natural
projection map from $X\times\Omega$ to $\Omega$ and $p_*(T)$ is the
pushforward of a differential form or current. When $T$ is a smooth
form this is the fiberwise integral of $T$. 
\begin{thm} Let $\Theta$ be the curvature form on $E$ and let $u_t$ be
  a holomorphic section of $E$. For each $t$ in
  $\Omega$ let $v_t$ solve (3.1) and be such that $\dbar_X
  v_t\wedge \omega=0$. 
Put
$$
\hat u=u_t-dt\wedge v_t.
$$
Then
\be
\langle\Theta u_t,u_t\rangle_t= p_*(c_n i\ddbar_{t,X}\phi \wedge \hat
u\wedge\overline{\hat u}\, e^{-\phi})  
+\int_X \|\dbar v_t\|^2 e^{-\phi_t}idt\wedge d\bar t.
\ee
\end{thm}

\begin{preremark} This formula shows that the curvature is nonnegative if $i\ddbar_{t,x}\phi\geq 0$. When $L=-K_X$ this implies immediately Theorem 1.1, for smooth curves, and the general case follows by the approximation techniques in the next section. The formula
can be found at the end of section 2.1 in \cite{3Berndtsson}. The proof there is a bit complicated since it  deals with the case of a general smooth proper fibration. In the present case, the proof follows from a computation of  
$$
\ddbar_t \,p_*(\hat u\wedge\overline{\hat u}).
$$
At least when $-K_X>0$, Theorem 1.1 can also be proved by differentiating $\F(t)$ and applying H\"ormander's $L^2$-estimate for $\dbar$. There are some difficulties in adapting this method of proof to the case when $-K_X$ is merely semipositive. However, the main advantage of using formula (3.2) instead is that it is useful when studying when equality holds in the inequality $\F''(t)\geq 0$, which we shall do next.
\end{preremark}

\bigskip

If the curvature acting on $u_t$ vanishes it follows that both terms
in the right hand side of (3.2) vanish. In particular, $v_t$ must be a
holomorphic form. To continue from there we first assume (like in
\cite{2Berndtsson})  that
$i\ddbar\phi_t>0$ on $X$. Taking $\dbar$ of formula 3.1 for $t$ fixed, we get 
$$
\dbar\dt v_t=\dbar\dof\wedge u_t.
$$
Using
$$
\dbar\dt +\dt\dbar=\ddbar\phi_t
$$
we get if $v_t$ is holomorphic that
$$
\ddbar\phi_t\wedge v_t=\dbar\dof\wedge u_t.
$$
The complex gradient of the function $i\dof$ with respect to the K\"ahler
metric $i\ddbar\phi_t$ is the $(1,0)$-vector
field defined by
$$
V_t\rfloor i\ddbar\phi_t=i\dbar\dof.
$$
Since $\ddbar\phi_t\wedge u_t=0$ for bidegree reasons we get
\be
\ddbar\phi_t\wedge v_t=\dbar\dof\wedge
u=(V_t\rfloor\ddbar\phi_t)\wedge u=
-\ddbar\phi_t\wedge (V_t\rfloor u).
\ee
If $i\ddbar\phi_t>0$ we find that
$$
-v_t=V_t\rfloor u.
$$
If $v_t$ is holomorphic it follows that  $V_t$ is a holomorphic vector
field - outside of the zerodivisor of $u_t$ and therefore everywhere
since the complex gradient is smooth under our hypotheses. If we
assume that $X$ carries no nontrivial holomorphic vector fields, $V_t$
and hence $v_t$ must vanish so $\dof$ is holomorphic, hence constant. 
Hence
$$
\ddbar\dof=0
$$
so $\ddbar\phi_t$ is independent of $t$. 
In general - if there are nontrivial holomorphic vector fields - we
get that the Lie derivative of $\ddbar\phi_t$ equals
$$
L_{V_t}\ddbar\phi_t=\partial
V_t\rfloor\ddbar\phi_t=\ddbar\dof=\frac{\partial}{\partial
  t}\ddbar\phi_t .
$$
Together with an additional argument showing that $V_t$ must be
holomorphic with respect to $t$ as well (see below) this gives that
$\ddbar\phi_t$ moves with the flow of the holomorphic vector field
which is what we want to prove.

For this  it is essential that the metrics $\phi_t$ be
strictly positive on $X$ for $t$ fixed, but we shall now see that
there is a way to get around this difficulty, at least in some special
cases.

The main case that we will consider is when the canonical bundle
of $X$ is seminegative, so we can take $L=-K_X$. Then $K_X+L$ is the
trivial bundle and we fix  a nonvanishing trivializing section
$u=1$. Then the constant section $t\rightarrow u_t=u$ is  a trivializing
section of the (line) bundle $E$. We write
$$
\F(t)=-\log \|u\|_t^2=-\log\int_X |u|^2 e^{-\phi_t}=-\log\int_X e^{-\phi_t}.
$$
Still assuming that $\phi$ is smooth, but perhaps not strictly
positive on $X$, we can apply  the curvature formula in Theorem 3.1
with $u_t=u$
and get
$$
\|u_t\|^2_t i\ddbar_t\F=\langle\Theta u_t,u_t\rangle_t= p_*(c_n
i\ddbar\phi \wedge \hat 
u\wedge\overline{\hat u} e^{-\phi_t})
+\int_X \|\dbar v_t\|^2 e^{-\phi_t}idt\wedge d\bar t.
$$
If $\F$ is harmonic, the curvature vanishes and it follows that $v_t$
is holomorphic on $X$ for any $t$ fixed. Since $u$ never vanishes we
can {\it define} a holomorphic vector field $V_t$ by
$$
-v_t=V_t\rfloor u.
$$
Almost as before we get 
$$
\dbar\dof\wedge u=\ddbar\phi_t\wedge v_t=-\ddbar\phi_t\wedge (V_t\rfloor u)
=(V_t\rfloor\ddbar\phi_t)\wedge u,
$$
which implies that
$$
V_t\rfloor i\ddbar\phi_t=i\dbar\dof.
$$
if  $\mathbf{u}$ never vanishes. This is the important point; we have been able
to trade the nonvanishing of $i\ddbar\phi_t$ for the nonvanishing of
$u$. This is where we use that the line bundle we are dealing with is
$L=-K_X$ (see section 9 for partial results for other line bundles).

We also get 
the formula for  the Lie derivative of
$\ddbar\phi_t$ along $V_t$
\be
L_{V_t}\ddbar\phi_t=\partial
V_t\rfloor\ddbar\phi_t=\ddbar\dof=\frac{\partial}{\partial
  t}\ddbar\phi_t .
\ee
To be able to conclude from here we also need to prove that $V_t$
depends holomorphically on $t$. For this we will use the first term in
the curvature formula, which also has to vanish. It follows that 
$$
i\ddbar\phi \wedge \hat
u\wedge\overline{\hat u}
$$
has to vanish identically. Since this is a semidefinite form in $\hat
u$ it follows that
\be
\ddbar\phi \wedge \hat u=0.
\ee

Considering the part of this expression that contains $dt\wedge d\bar
t$ we see that
\be
\mu:=\frac{\partial^2\phi}{\partial t\partial \bar
  t}-\partial_X(\frac{\partial\phi}{\partial\bar t})(V_t)=0.
\ee

\bigskip

If $\ddbar_X\phi_t>0$, $\mu$ is easily seen to be equal to the
function $c(\phi)$ defined in the introduction, so the vanishing of
$\mu$ is then equivalent to the homogenous Monge-Amp\`ere equation.
In \cite{2Berndtsson} we showed that $\partial V_t/\partial \bar t=0$
by realizing this vector field as the complex gradient of the function
$c(\phi)$ which has to vanish if the curvature is zero. Here, where we
no longer assume strict postivity of $\phi_t$ along $X$ we have the
same problems as earlier to define the complex gradient. Therefore we
follow the same route as before, and start by studying  $\partial
v_t/\partial \bar t$ instead.

\bigskip

Recall that
$$
\dt v_t=\dof\wedge u +h_t
$$
where $h_t$ is holomorphic on $X$ for each $t$ fixed. As we have seen
in the beginning of this section, $v_t$ is uniquely determined, and it
is not hard to see that it depends smoothly on $t$ if $\phi$ is
smooth. Differentiating
with respect to $\bar t$ we obtain
$$
\dt\frac{\partial v_t}{\partial \bar t}=\left [\frac{\partial^2\phi}{\partial
  t\partial \bar t}-\partial_X(\frac{\partial\phi}{\partial\bar
  t})(V_t)\right ]\wedge u +\frac{\partial h_t}{\partial\bar t}.
$$
Since the left hand side is automatically orthogonal to holomorphic
forms, we get that
$$
\dt\frac{\partial v_t}{\partial \bar t}=\pi_\perp(\mu u)=0,
$$
since $\mu=0$ by (3.6). Again, this means that $\partial
v_t/\partial \bar t =0$ since $\partial v_t/\partial \bar
t\wedge\omega$ is still $\dbar_X$-closed, and 
the cohomological assumption implies
that $\dt$ is injective on $(n-1,0)$-forms $\gamma$ such that $\gamma\wedge\omega$ is $\dbar$-closed.  

All in all, $v_t$ is holomorphic in $t$, so $V_t$ is holomorphic on
$X\times\Omega$. Let $F_t$ be the flow of the time dependent holomorphic vector field $-V_t$, so that for any function $\psi$ on $X$
$$
\frac{\partial}{\partial t} \psi(F_t(z))=-V_t(\psi)(F_t(z)).
$$
Then we also have for any form $\eta$ on $X$ that
$$
\frac{\partial}{\partial t}F_t^*(\eta)=-F_t^*(L_{V_t}\eta).
$$
Applying this to $\eta=i\ddbar_X\phi_t$ we get
$$
\frac{\partial}{\partial t}F_t^*(i\ddbar\phi_t)=F_t^*(\frac{\partial}{\partial t}i\ddbar\phi_t-L_{V_t}i\ddbar\phi_t)=0
$$
by (3.4). Since $\eta$ is real form, we can take real and imaginary parts of this, so $F_t^*(\ddbar\phi_t)=\ddbar\phi_0$ which completes the proof.

\section{ The nonsmooth case}

Our strategy to treat the general case is to  write our bounded curve of metrics  $\phi=\phi_t$ as
the decreasing limit of a  
sequence of smooth  metrics, $\phi^\nu$, with $i\ddbar\phi^\nu\geq
-\epsilon_\nu\omega$, where $\epsilon_\nu$ tends to zero, see section
2.3. Then we can apply Theorem 3.1 for the metrics $\phi^\nu$ and study the limit as $\nu$ tends to infinity. Note also 
that in case we assume that $-K_X>0$ we can even 
approximate with metrics of strictly positive curvature. The presence
of the negative term $-\epsilon_\nu \omega$ causes some minor
notational  problems
in the estimates below. We will therefore carry out the proof under
the assumptions that  $i\ddbar\phi^\nu\geq 0$ and leave the necessary
modifications to the reader. Throughout in this section we assume that $\phi_t$ depends only on the real part of $t$. Thus we also assume that $\Omega= I\times {i\R}$ is a strip, and for ease of notation we assume that $0$ lies in the interval $I$, so that zero is an interior point of $\Omega$. 

\bigskip

\noindent Let $\F_\nu$ be
defined the same way as $\F$, but using the weights $\phi^\nu$
instead. Then 
$$
i\ddbar\F_\nu
$$
goes to zero weakly on $\Omega$. We  get a sequence of $(n-1,0)$
forms $v^\nu_t$, solving
$$
\dt v^\nu_t=\pi_\perp(\dot{\phi_t^\nu} u)
$$
for $\phi=\phi_\nu$. By Remark 1, we have an $L^2$-estimate for
$v^\nu_t$ in terms of the $L^2$ norm of $\dofnu$, with the constant in
the estimate independent of $t$ and $\nu$. Since $\dot{\phi_t^\nu}$ is
uniformly bounded by section 2.2, it follows that we get a
uniform bound for the $L^2$-norms of $v_t^\nu$ over all of
$X\times\Omega$. Therefore we can select a 
subsequence of $v_t^\nu$ that converges weakly to a form $v$ in
$L^2$. Since $i\ddbar\F_\nu$ tends to zero weakly, Theorem 3.1 shows
that the $L^2$-norm of $\dbar_X v^\nu$ over $X\times K$ goes to zero
for any compact $K$ in $\Omega$, so $\dbar_X v=0$. Moreover, we claim that
$$
\dt_X v= \pi_\perp (\dot{\phi}_tu)
$$
in the (weak ) sense that 
\be
\int_{X\times\Omega}dt\wedge d\bar t\wedge v\wedge\overline{\dbar W} e^{-\phi}=
(-1)^n\int_{X\times\Omega}dt\wedge d\bar t\wedge
\pi_\perp(\dot{\phi}_t u)\wedge\overline{ W }e^{-\phi} 
\ee
for any smooth form $W$ of the appropriate degree.

To see this, note first that
$$
 \int_{X\times\Omega}dt\wedge d\bar t\wedge v^\nu\wedge\overline{\dbar W} e^{-\phi^\nu}=\lim (-1)^n\int_{X\times\Omega}dt\wedge d\bar t\wedge
\pi_\perp(\dot{\phi^\nu}_tu)\wedge\overline{ W }e^{-\phi^\nu} .
$$
In the left hand side we then use that (a subsequence of) $v^\nu$ converges weakly in $L^2$ (since the metrics $\phi^\nu$ are bounded we don't need to worry about which $L^2$) to $v$. By dominated convergence we also have that $\dbar W e^{-\phi^\nu}$ converges strongly to $\dbar W e^{-\phi}$. Combining these two facts we see that the left hand side converges to 
$$
\int_{X\times\Omega}dt\wedge d\bar t\wedge v\wedge\overline{\dbar W} e^{-\phi}.
$$
As for the right hand side we decompose 
$$
\pi_\perp(\dot{\phi^\nu}_t u)=\dot{\phi^\nu}_t u+h^\nu
$$
where $h^\nu$ is holomorphic and both terms are bounded in $L^2$. We can then take limits in the same way and find that the right hand side tends to
$$
\int_{X\times\Omega}dt\wedge d\bar t\wedge
(\dot{\phi}_tu+h)\wedge\overline{ W }e^{-\phi}. 
$$
A similar argument then shows that $\dot{\phi}u+h$ is orthogonal to holomorphic forms and so must equal $\pi_\perp(\dot{\phi}u) $ which completes the proof of (4.1). 

\bigskip

Formula (4.1) says that in the sense of distributions
$$
\partial_X (ve^{-\phi})=\pi_\perp(\dot{\phi}_tu) e^{-\phi}
$$
(in a local trivialization). We next claim that this means that
$$
\partial_X v -\partial_X\phi\wedge v=\pi_\perp(\dot{\phi}_tu).
$$
This  is because in the sense of distributions
$$
\partial_X(v e^{-\phi})=\lim \partial_X(v e^{-\phi^\nu})=\lim (\partial_X v-\partial_X\phi^\nu\wedge v)e^{-\phi^\nu},
$$
which equals 
$$
 (\partial_X v-\partial_X\phi\wedge v)e^{-\phi}
$$
by essentially the same argument as before. 

\bigskip

We can now take $\dbar_X$ of this equation and find that
\be
\ddbar_X\phi\wedge v=\dbar_X\dot{\phi}_t\wedge u.
\ee
Just as in the previous section we then define a $t$ dependent vector field on $X$ by
$$
V\rfloor u=v.
$$
Since $\dbar_X v=0$, $V$ is holomorphic on $X$ for $t$ fixed, and satisfies as before that
$$
V\rfloor\ddbar_X\phi=\dbar_X\dot{\phi}.
$$

\bigskip

As before this ends the argument if there are no nontrivial holomorphic vector
fields on $X$. Then $v$ must be zero, so $\dof$ is holomorphic, hence
constant. In the general case, we finish by showing that $v_t$ is
holomorphic in $t$. The difficulty is that we don't know  any
regularity of $v_t$ with respect to $t$, 
except that it lies in $L^2$, so we need to formulate 
holomorphicity weakly.

We will use two elementary lemmas that we state
without proof. The first one allows us get good convergence
properties for geodesics, when the metrics only depend on the real
part of $t$ and therefore are convex with respect to $t$.

\begin{lma}
Let $f_\nu$ be a sequence of smooth convex functions on an interval in $\R$
that decrease to the convex function $f$. Let $a$ be a point
in the interval such that $f'(a)$ exists. Then $f_\nu'(a)$ converge
to $f'(a)$. Since a convex function is differentiable almost
everywhere it follows that $f'_\nu$ converges to $f'$ almost
everywhere, with dominated convergence on any compact subinterval.
\end{lma}

In particular the lemma can be applied to a decreasing sequence $\phi^\nu_t$ of subgeodesics that are independent of $\Im t$ and decrease to a geodesic $\phi_t$. For any fixed $x$ in $X$ it follows that $\dot{\phi}^\nu_t(x)$ converges to $\dot{\phi}_t(x)$ for almost all $t$, so it follows that this holds almost everywhere on $\Omega\times X$. By section 2.3  we also have a fixed bound on the $t$-derivative of $\phi_t^\nu$, so  we even  have dominated convergence. 

Another technical problem that arises is that we are dealing with
certain orthogonal projections on the manifold $X$, where the weight
depends on $t$. The next lemma gives us control of how these
projections change.
\begin{lma} Let $\alpha_t$ be forms on $X$ with coefficients depending
  on $t$ in $\Omega$. Assume that $\alpha_t$ is Lipschitz with respect to $t$ as
  a map from $\Omega$ to $L^2(X)$. Let $\pi^t$ be the orthogonal
  projection on $\dbar$-closed forms with respect to the metric $\phi_t$ and the
  fixed K\"ahler metric $\omega$. Then $\pi^t(\alpha_t)$ is also
  Lipschitz, with a Lipschitz constant depending only on that of
  $\alpha$ and the Lipschitz constant of $\phi_t$ with respect to $t$.
\end{lma}
Note that in our  case, when $\phi$ is independent of the imaginary
part of $t$, we 
have control of the Lipschitz constant with respect to $t$ of $\phi_t$ , and also
by the first 
lemma uniform control of the Lipschitz constant of $\phi^\nu_t$, since
the derivatives are increasing.

\bigskip

It follows from the curvature formula that 
$$
a_\nu:=\int_{X\times\Omega'}i\ddbar\phi^\nu\wedge\hat u\wedge\overline{\hat
  u}e^{-\phi^\nu}
$$
goes to zero if $\Omega'$ is a relatively compact subdomain of
$\Omega$. Shrinking $\Omega$ slightly we assume that this actually
holds with $\Omega'=\Omega$. By the Cauchy inequality
$$
\int_{X\times\Omega}i\ddbar\phi^\nu\wedge\hat
u\wedge\overline{W}e^{-\phi^\nu}\leq
(a_\nu\int_{X\times\Omega}i\ddbar\phi^\nu\wedge W\wedge\bar W e^{-\phi^\nu})^{1/2}
$$
if $W$ is any $(n,0)$-form. Choose $W$ to contain no
differential $dt$, so that it is an $(n,0)$-form on $X$ with
coefficients depending on $t$. Then
$$
\int_{X\times\Omega}i\ddbar\phi^\nu\wedge W\wedge\bar W e^{-\phi^\nu}=
\int_{X\times\Omega}i\ddbar_t\phi^\nu\wedge
W\wedge\bar W e^{-\phi^\nu}
$$
We now assume that $W$ has compact support. We will then use the one variable H\"ormander inequality, which says that if $w$ is a function of $t$ with compact support in $\Omega$ and one derivative in $L^2$,  and $\psi$ is a smooth function in $\Omega$ , then
$$
\int_\Omega i\ddbar \psi |w|^2 e^{-\psi}\leq \int_\Omega |\partial^\psi w|^2 e^{-\psi}.
$$
(This is the dual version of H\"ormander's $L^2$-estimate and can be found in \cite{Hormander}.)
We apply this inequality to $w=W$ and $\psi=\phi^\nu$, where we consider $W$ and $\phi^\nu$ as  functions of $t$ by holding the $X$-variable fixed.
The one variable
H\"ormander inequality with respect to $t$ then shows that
\be
\int_{X\times\Omega}i\ddbar_t\phi^\nu\wedge
W\wedge\bar W e^{-\phi^\nu}\leq\int_{X\times\Omega}|\partial^{\phi^\nu}_t W|^2e^{-\phi^\nu}.
\ee
From now we assume that $W$ is Lipschitz with respect to $t$ as a map from
$\Omega$ into $L^2(X)$. Then (4.3) is uniformly bounded, so

$$
\int_{X\times\Omega}idt\wedge d\bar
t\wedge(\mu^\nu u)\wedge\overline{W}e^{-\phi^\nu} 
$$
goes to zero, where $\mu^\nu$ is defined as in (3.6) with $\phi$ replaced by
$\phi^\nu$. By Lemma 4.2
$$
\int_{X\times\Omega}idt\wedge d\bar
t\wedge(\mu^\nu u)\wedge\overline{\pi_\perp W}e^{-\phi^\nu} 
$$
also goes to zero. Therefore
$$
\int_{X\times\Omega}idt\wedge d\bar
t\wedge\pi_\perp(\mu^\nu u)\wedge\overline{W}e^{-\phi^\nu}. 
$$
goes to zero. Now recall that
$\pi_\perp(\mu^\nu u)=\partial^{\phi_\nu} (\partial v_t^\nu/\partial\bar t)$ and integrate
by parts. This gives that
$$
\int_{X\times\Omega}idt\wedge d\bar
t\wedge \frac{\partial v_t^\nu}{\partial \bar
  t}\wedge\overline{\dbar_XW}e^{-\phi^\nu}  
$$ 
also vanishes as $\nu$ tends to infinity. 

\bigskip

Next we let $\alpha$ be a form of bidegree $(n,1)$ on $X\times\Omega$
that does not contain any differential $dt$. We assume it is Lipschitz
with respect to $t$ and decompose it into one part, $\dbar_X W$,  which is
$\dbar_X$-exact and 
one which is orthogonal to $\dbar_X$-exact forms. This amounts of
course to making this orthogonal decomposition for each $t$
separately, and by Lemma 4.2 each term in the decomposition is still
Lipschitz in $t$, uniformly in $\nu$. Since $v^\nu_t\wedge\omega$ is
$\dbar_X$-closed 
by construction, this holds also for $\partial v^\nu/\partial \bar
t$. By our cohomological assumption, it is also $\dbar$-exact,
and we get that
$$
\int_{X\times\Omega}idt\wedge d\bar
t\wedge \frac{\partial v_t^\nu}{\partial \bar
  t}\wedge\overline\alpha e^{-\phi^\nu}=  \int_{X\times\Omega}idt\wedge d\bar
t\wedge \frac{\partial v_t^\nu}{\partial \bar
  t}\wedge\overline{\dbar_XW}e^{-\phi^\nu}.  
$$ 
Hence

$$
\int_{X\times\Omega}dt\wedge
v_t^\nu\wedge\overline{\partial^{\phi^\nu}_t\alpha}e^{-\phi^\nu}
$$
goes to zero. 
By Lemma 4.1 we may pass to the limit here and finally get 
that
\be
\int_{X\times\Omega} dt\wedge
v_t\wedge\overline{\partial^{\phi}_t\alpha}e^{-\phi} =0,  
\ee
under the sole assumption that $\alpha$ is of compact support, and
Lipschitz in $t$. This is almost the distributional formulation of
$\dbar_t v=0$, except that $\phi$ is not smooth. But, replacing
$\alpha$ by $e^{\phi-\psi}\alpha$, where $\psi$ is another metric on
$L$, we see that  if (4.4) holds for some $\phi$, Lipschitz in $t$, it
holds for any such metric. Therefore we can replace $\phi$ in (4.4) by
some other smooth metric. It follows that $v_t$ is holomorphic in $t$
and therefore, since we already know it is holomorphic on $X$,
holomorphic on $X\times\Omega$. This completes the proof.

\subsection{Time independence of $V$}
We shall now prove that the vector fields $V_t$ are in fact independent of time $t$, i e that $V_t= (F_t)_*(V_0)$. Let $\V= \partial/\partial t-V_t$. This is a holomorphic vector field on $\Omega\times X$. 
\begin{lma}
\be
\V\rfloor \ddbar_{t,X}\phi=0.
\ee
\end{lma}
\begin{proof}
Recall that  $F_t$ is the flow of the time dependent holomorphic vectorfield $-V_t$, and that 
\be
F_t^*(\ddbar_X\phi_t)=\ddbar\phi_0
\ee
 (see end of section 3). Moreover, if $\psi(t,z)$ is a function then
$$
(\partial/\partial t)\psi(t, F_t(z))=\V\psi.
$$
By (4.6), the volume forms $e^{-\phi_t}$ must  satisfy
$$
F_t^*(e^{-\phi_t})=e^{-\phi_0 +c(t)},
$$
with $c(t)$ constant on $X$ for $t$ fixed. Integrating over $X$ we find that
$$
\log\int e^{-\phi_t}=c(t) +\log\int e^{-\phi_0},
$$
so $c(t)$ is by assumption a linear function. Choose local coordinates $z^j$ and take representatives of the metrics, $\phi_t^j$. Then 
$$
\V\phi^j_t(z)=(d/dt)(\phi_0(z) -c(t))=-c'.
$$
Hence 
$$
\V\rfloor\ddbar_{t,X}\phi_t=\dbar\V\phi_t^j=0.
$$
\end{proof}
\begin{preremark} The form $\hat{u}= u-dt\wedge v$ in Theorem 3.1 can be written $\hat{u}=-\V\rfloor (dt\wedge u)$. Using this one can check that the equation 
$\V\rfloor \ddbar\phi=0$ is equivalent to $\ddbar\phi\wedge\hat{u}=0$. This means precisely that the first term in the curvature formula (3.2) vanishes. We have given the indirect proof above to avoid having to check that the formula (3.2) holds in the limit as well.
\end{preremark}
\begin{lma} 
$(F_t)^*(\dof)$ is independent of $t$.
\end{lma}
\begin{proof} Since $\phi$ depends only on $\Re t$, $\dof$ is real valued, so it suffices to prove that $(F_t)^*(\dof)$ is holomorphic in $t$. But
$$
\partial/\partial \bar t (F_t)^*(\dof)=\bar\V(\dof),
$$
and since $\bar\V(\dof)$ is the coefficient of $dt$ in $\bar\V\rfloor\ddbar_{t,X}\phi$ it vanishes by the previous lemma.
\end{proof}
\begin{prop} $V_t$ is a time independent vector field, i e  $$V_t= (F_t)_*(V_0)$$. 
\end{prop}
\begin{proof} We know that 
$$
V_t\rfloor\omega_t=\dbar\dof.
$$
Pulling back by the biholomorphic map $F_t$ we get from the previous lemma (since $F_t^*(\omega_t)=\omega_0$) that
$$
 (F_{-t})_*(V_t)\rfloor\omega_0=\dbar\dot{\phi}_0.
$$
This means that 
$$
 ((F_{-t})_*(V_t)-V_0)\rfloor\omega_0=0.
$$
In case $\omega_0$ is smooth and strictly positive this implies immediately that
$(F_{-t})_*(V_t)=V_0$. In the general case we conclude by Proposition 8.2 (see section 8).
\end{proof}
  
\bigskip
\section{ The Bando-Mabuchi theorem.}
A K\"ahler metric, $\omega$,  on a Fano manifold $X$ is a K\"ahler-Einstein metric if it equals a constant multiple of its Ricci form, i e if it satisfies the equation 
\be
Ric(\omega)= a\omega,
\ee
where $a$ is a positive constant. Multiplying $\omega$ with the constant $a$ does not change the Ricci curvature so we may always assume that $a=1$, and then $\omega$ must lie in $c_1(-K_X)$. This means that $\omega=i\ddbar\phi$ for some metric $\phi$ on $-K_X$ and (5.1)  says that
$$
e^{-\phi}= a'(i\ddbar\phi)^n,
$$
for some constant $a'$ if we interpret the left hand side as a volume form as described in section 2.1. This means that $\phi$ is a critical point for the {\it Ding functional}
$$
D(\psi):=\log \int e^{-\psi}+\E(\psi,\psi_0)/\text{Vol}(-K_X),
$$
where $\psi_0$ is an arbitrary metric on $-K_X$ and $\E$ is the relative Monge-Ampere energy (see section 2.4 for definition and basic properties). 
 Thus $\phi_0$ solves the K\"ahler-Einstein
equation if and only if $(d/ds|_{s=0}D(\phi_s)=0$  for any smooth curve $\phi_s$.

Suppose now that
$\phi_0$ and $\phi_1$ are two K\"ahler-Einstein metrics. We connect them
by a bounded geodesic $\phi_t$ . Then
$\phi_t$ depends only on the real part of $t$ so $\G(t):=-D(\phi_t)$ is convex. We
claim that since
both end points are K\"ahler-Einstein metrics, 0 and 1 are stationary
points for $\G$, so $\G$ must be linear in $t$.  This would be immediate if the geodesic were smooth,
but we claim that it also holds if the geodesic is only
bounded, with boundary behaviour as described in section 2.2.  The
function $\F$ is convex, hence has onesided derivatives at the
endpoints, and using the convexity of $\phi$ with respect to $t$ one
sees that they equal
$$
\int \dof e^{-\phi}/\int e^{-\phi}
$$
(where $\dof$ now stands for the onesided derivatives). For the function
$\E(\phi_t,\psi)$ we use the inequalities (2.3) and (2.4). They show that the onesided derivatives of $\G$ at $t=0$ and $t=1$ satisfy $\G'(0)\geq 0$ and 
$\G'(1)\leq 0$. Since $\G$ is convex this is only possible if both derivatives are zero and $\G$ is constant. As moreover $\E$ is affine along the geodesic it follows that $\log\int e^{-\phi_t}$ is also affine.

Thus we can apply Theorem 1.2 and it follows that  $\ddbar\phi_t$ are related via the flow of a holomorphic vector field, so we have proved the following theorem of Bando and Mabuchi, \cite{Bando-Mabuchi}. 
\begin{thm} Let $X$ be a Fano manifold and suppose $\omega_0$ and $\omega_1$ are two solutions of the K\"ahler-Einstein equation (5.1). Then there is a holomorphic vector field on $X$ with time 1 flow $F$, such that $F^*(\omega_1)=\omega_0$.
\end{thm}
Notice that in our proof we do not need to assume that the metrics are smooth -it is enough to assume that their potentials are bounded.

\section{Two extensions of Theorem 1.2 for unbounded metrics}

One might ask if Theorem 1.2 is valid under even more general
assumptions. A minimal requirement is of course that $\F$ be finite,
or in other words that $e^{-\phi_t}$ be integrable. For all we know
Theorem 1.2 might be true in this generality, but here we
will limit ourselves to   curves of metrics that can be decomposed into one part which is bounded and an unbounded part that does not depend on $t$.  

\subsection{The case $-K_X\geq 0$}

\bigskip

\noindent
 Let $t\rightarrow \phi_t$ be a curve of singular metrics
  on $L=-K_X\geq 0$ that can be written
$$
\phi_t=\tau_t +\psi
$$
where $\psi$ is a metric on an $\R$-line bundle $S$ and $\tau_t$ is a
curve of metrics on $-(K_X+S)$ such that:

\medskip

(i) $\tau_t$ is bounded  and
only depends on $\Re t$.

\medskip

(ii) $e^{-\psi}$ is integrable, $\psi$ does not depend on $t$ and $\psi$ is locally bounded in the complement of a closed pluripolar set. 

and

\medskip

(iii) 
$i\ddbar_{t,X}(\phi_t)\geq 0$. 

\medskip

\begin{thm} Assume that $-K_X\geq 0$ and that $H^{0,1}(X)=0$. Let
  $\phi_t=\tau_t+\psi$ be a 
  curve of metrics on $-K_X$ 
  satisfying  (i)-(iii). Assume that
$$
\F(t)=-\log\int_X e^{-\phi_t}
$$
is affine. Then there is a holomorphic vector field $V$ on $X$ with
flow $F_t$ such that
$$
F_t^*(\ddbar\phi_t)=\ddbar\phi_0.
$$
\end{thm}

We also state an important addendum.
\begin{thm} Assume that in addition to the assumptions in Theorem 6.1 
$i\ddbar\psi\geq 0$ and $i\ddbar\tau_t\geq 0$ in the sense of currents. Then 
\be
V\rfloor i\ddbar\psi=0
\ee
and $F_t^*(\ddbar\psi)$ is independent of $t$.
\end{thm}
We shall see in the last section that in many cases (6.1) implies that actually $V=0$, so that the flow $F_t^* $ is the identity map and $\ddbar\phi_t$ must also be independent of $t$.

The proof of this theorem is almost the same as the proof of Theorem
1.2. The main thing to be checked is that for $\phi=\phi^\nu$ a
sequence of smooth metrics decreasing to $\phi$ we can still solve the
equations 
$$
\partial^{\phi_t^\nu} v_t=\pi_\perp (\dot{\phi_t^\nu} u)
$$
with an $L^2$ -estimate independent of $t$ and $\nu$. 
\begin{lma}
Let $L$ be a holomorphic line bundle over $X$ with a metric $\xi$
satisfying $i\ddbar\xi\geq- \omega_0$ for some fixed K\"ahler form $\omega_0$. Let $\xi_0$ be a smooth metric on $L$
with $\xi\leq \xi_0$, and assume
$$
I:=\int_X e^{\xi_0-\xi}<\infty.
$$
Then there is a constant $A$, only depending on $I$ and $\xi_0$ (
not on $\xi$!) such that if $f$ is a $\dbar$-exact $L$ valued
$(n,1)$-form with 
$$
\int_X |f|^2 e^{-\xi}\leq 1
$$
there is a solution $u$ to $\dbar u=f$ with 
$$
\int_X |u|^2 e^{-\xi}\leq A.
$$
(The integrals are understood to be taken with respect to some
arbitrary smooth volume form.)
\end{lma}  
\begin{proof}
The assumptions imply that
$$
\int |f|^2 e^{-\xi_0}\leq 1.
$$
Since $\dbar$ has closed range for $L^2$-norms defined by smooth
metrics, we can solve $\dbar u=f$ with
$$
\int |u|^2 e^{-\xi_0}\leq C
$$
for some constant depending only on $X$ and $\xi_0$. Choose a
collection of coordinate balls $B_j$ such that $B_j/2$ cover $X$. In
each $B_j$ we can, by classical H\"ormander estimates in $\C^n$, solve $\dbar u_j=f$ with
$$
\int_{B_j} |u_j|^2 e^{-\xi}\leq C_1\int_{B_j} |f|^2 e^{-\xi}\leq C_1,
$$
$C_1$ only depending on the size of the balls and the choice of $\omega_0$. Then $h_j:=u-u_j$ is
holomorphic on $B_j$ and 
$$
\int_{B_j}|h_j|^2 e^{-\xi_0}\leq C_2,
$$
so
$$
\sup_{B_j/2}|h_j|^2 e^{-\xi_0}\leq C_3.
$$
Hence
$$
\int_{B_j/2}|h_j|^2 e^{-\xi}\leq C_3 I
$$
and therefore
$$
\int_{B_j/2}|u|^2 e^{-\xi}\leq C_4 I.
$$
Summing up we get the lemma.
\end{proof}

By the discussion in section 2.3, the assumption that $-K_X\geq 0$
implies that we can write $\phi_t$ as a limit of a decreasing sequence
of smooth metrics $\phi_t^\nu$ with
$$
i\ddbar\phi_t^\nu\geq -\epsilon_\nu \omega
$$
where $\epsilon_\nu$ tends to zero. 
Applying the lemma  to $\xi=\phi^\nu_t$ and $\xi_0$ some arbitrary smooth
metric we see that we have uniform estimates for solutions of the
$\dbar$-equation, independent of $\nu$ and $t$. By remark 2, section
3, the same holds for the adjoint operator, which means that we can
construct $(n-1,0)$-forms $v^\nu$ just as in section 3, and have a uniform bound on their $L^2$-norms. Again we can take weak limits and get an $(n-1,0)$-form, $v$ that satisfies $\dbar_X v=0$. 

\bigskip

We claim that $v$ satisfies formula 
\be
\partial_X v-\partial_X\phi\wedge v=\pi_\perp(\dot{\phi}u)
\ee
as in the case of bounded metrics in section 4. This is not quite obvious since the proof of this rested on (4.1) which used that the geodesic was bounded. However, formula (4.1) still holds if $W$ is supported outside the closed pluripolar set where $\psi=-\infty$. This means that (6.2) holds there. Moreover, the left hand side lies in $L^2$ with respect to our unbounded metric, hence in particular in ordinary $L^2_{loc}$. Formula  (6.2) says that (locally) 
$$
\partial_X v-\partial_X\phi\wedge v-\dot{\phi} u
$$
is holomorphic on $X$ for fixed $t$ away from the singular set of $\psi$. Since it is moreover in $L^2$ and pluripolar sets are removable for $L^2$-holomorphic functions it follows that it is holomorphic on all of $X$.

Hence we conclude that
$$
\ddbar_X\phi\wedge v=\dbar\dot{\phi}\wedge u
$$
on all of $X$. 
We then again define a vector field $V$ on $X$ by $V\rfloor u= v$ and find that
\be
V\rfloor \ddbar_X\phi=\dbar_X\dot{\phi}
\ee
and the proof concludes in  the same way as before. 

We finally turn to the proof of the addendum in Theorem 6.2. For this we use the  vector field on $\Omega\times X$
$$
\mathcal{V}:=\frac{\partial}{\partial t}-V,
$$
as in section 4.1.

Lemma 4.3 implies that 
$$
0=i\bar{\mathcal{V}}\wedge\mathcal{V}\rfloor\ddbar_{t,X}\phi=i\bar{\mathcal{V}}\wedge\mathcal{V}\rfloor\ddbar_{t,X}\tau +i\bar{\mathcal{V}}\wedge\mathcal{V}\rfloor\ddbar\psi
$$
Since both terms in the right hand side are nonnegative by assumption, they must both vanish. But, since $\psi$ does not depend on $t$
$$
+i\bar{\mathcal{V}}\wedge\mathcal{V}\rfloor\ddbar\psi=+i\bar V\wedge V\rfloor\ddbar\psi.
$$
Since $i\ddbar\psi$ is a positive current, this implies by Cauchy's inequality that $\bar W\wedge V\rfloor i\ddbar \psi=0$ for any $(0,1)$ vector field $W$, so $V\rfloor\ddbar\psi=0$ . This proves Theorem 6.2.
\subsection{Yet another version}
We also briefly describe yet another situation where the same
conclusion as in Theorem 6.1 can be drawn even though we do not assume
that $-K_X\geq 0$. The assumptions are very particular, and it is not
at all clear that they are optimal, but they are chosen to fit with
the properties of desingularisations of certain singular varieties. We
then assume instead that $-K_X$ can be 
decomposed
$$
-K_X = -(K_X +S) +S
$$
where $S$ is the $\R$-line bundle corresponding to a klt -divisor
$\Delta\geq 0$ and
we assume $-(K_X+S)\geq 0$. We moreover assume that the underlying
variety of $\Delta$ is a union of
smooth hypersurfaces with simple normal crossings. We then look at curves 
$$
\phi_t=\tau_t +\psi
$$
where $\tau_t$ is bounded, $i\ddbar_{t, X}\tau_t\geq 0$ and $\psi$ is a fixed  metric on $S$
satisfying $i\ddbar\psi=[\Delta]$. We claim that the conclusion of
Theorem 6.1 holds in 
this situation as well. The difference as compared to our previous
case is that we do not assume that $\phi_t$ can be approximated by a
decreasing sequence of metrics with almost positive curvature.  For
the proof we approximate $\tau_t$ by a 
decreasing sequence of smooth metrics $\tau^\nu_t$ satisfying
$$
i\ddbar\tau_t^\nu\geq -\epsilon_\nu \omega.
$$
As for $\psi$ we approximate it following the scheme at the end of
section 2.3 by a sequence satisfying
$$
i\ddbar\psi^\nu\geq -C\omega
$$
and 
$$
i\ddbar\psi^\nu\geq -\epsilon_\nu \omega
$$
outside of any neighbourhood of $\Delta$. Then let
$\phi^\nu_t=\tau^\nu_t+\psi^\nu$. Now consider the curvature
formula (3.2)
\be
\langle\Theta^\nu u_t,u_t\rangle_t= p_*(c_n i\ddbar\phi^\nu_t \wedge \hat
u\wedge\overline{\hat u} e^{-\phi^\nu_t})  
+\int_X \|\dbar v^\nu_t\|^2 e^{-\phi^\nu_t}idt\wedge d\bar t
\ee
We want to see that the second term in the right hand side tends to
zero given that the 
curvature $\Theta^\nu$ tends to zero, and the problem is that the
first term on the right hand side has a negative part. However,
$$
 p_*(c_n i\ddbar\phi^\nu_t \wedge \hat
u\wedge\overline{\hat u} e^{-\phi^\nu_t}) 
$$
can for any $t$ be estimated from below by 
\be
-\epsilon_\nu \|\hat u\|^2 -C\int_U |v^\nu_t|^2 e^{-\phi^\nu}
\ee
where $U$ is any small neighbourhood of $\Delta$ if we choose $\nu$
large. This means, first, that we still have at least a uniform upper
estimate on $\dbar v^\nu_t$. This, in turn gives by the technical lemma below
that the $L^2$-norm of $v^\nu_t$ over a small neighbourhood of
$\Delta$ must be small if the neighbourhood is small. Shrinking the
neighbourhood as $\nu$ grows we can then arrange things so that the
negative part in the right hand side goes to zero. Therefore the
$L^2$-norm of $\dbar v^\nu_t$ goes to zero after all, and the limit of 
$$
\langle\Theta^\nu u_t,u_t\rangle
$$
is zero. The last fact means that
$$
t\to -\log\int e^{-\phi_t}
$$
is convex.

After this
the proof proceeds as before. We collect this in the next theorem.
\begin{thm} Assume that $-(K_X+S)\geq 0$ and that $H^{0,1}(X)=0$. Let
  $\phi_t=\tau_t+\psi$ be a 
  curve of metrics on $-K_X$ 
where 

(i) $\tau_t$ are bounded metrics on $-(K_X+S)$ with $i\ddbar\tau_t\geq 0$, depending only on $\Re t$,

and 

(ii) $\psi$ is a metric on $S$ with $i\ddbar\psi=[\Delta]$, where
$\Delta$ is a klt divisor with simple normal crossings. 

Then 
$$
\F(t)=-\log\int_X e^{-\phi_t}
$$
is convex. If $\F(t)$ is affine,  there is a holomorphic vector field $V$ on $X$ with
flow $F_t$ such that
$$
F_t^*(\ddbar\phi_t)=\ddbar\phi_0.
$$
\end{thm}
We end this section with the technical lemma used above.
\begin{lma} The term
$$
\int_U |v^\nu_t|^2 e^{-\phi^\nu}
$$
in (6.2) can be made arbitrarily small if $U$ is a sufficiently small
neighbourhood of $\Delta$
\end{lma}
\begin{proof}
Covering $\Delta$ with a finite number of polydisks, in which the
divisor is a union of coordinate hyperplanes,  it is enough to prove
the following statement:

Let $P$ be the unit  polydisk in $\C^n$ and let $v$ be a compactly
supported function in $P$. Let 
$$
\psi_\epsilon=\sum \alpha_j\log( |z_j|^2 +\epsilon)
$$
where $0\leq \alpha_j<1$. Assume
$$
\int_P (|v|^2 +|\dbar v|^2)e^{-\psi}\leq 1.
$$
Then for $\delta>>\epsilon$
$$
\int_{\cup\{|z_j|\leq \delta\}} |v|^2 e^{-\psi_\epsilon} \leq c_\delta
$$
where $c_\delta$ tends to zero with $\delta$.

To prove this we first  estimate the integral over $ |z_1|\leq \delta$
 using the one 
variable Cauchy formula in the first variable
$$
v(z_1, z')= \pi^{-1}\int v_{\bar\zeta_1}(\zeta_1,z')/(\zeta_1 -z_1)
$$
which gives
$$
|v(z_1, z')|^2 \leq C \int |v_{\bar\zeta_1}(\zeta_1,z')|^2/|\zeta_1
-z_1|.
$$
Then multiply by $(|z_1|^2 +\epsilon)^{-\alpha_1}$ and integrate with
respect to $z_1$ over $|z_1|\leq \delta$. Use the estimate
$$
\int_{|z_1|\leq \delta} \frac{1}{(|z_1|^2
  +\epsilon)^{\alpha_1}|z_1-\zeta_1|}\leq  c_\delta
(|\zeta_1|^2+\epsilon)^{-\alpha_1},
$$
multiply by $\sum_2^n \alpha_j\log( |z_j|^2 +\epsilon)$ and integrate
with respect to $z'$. Repeating the same argument for $z_2, ..z_n$ and
summing up we get the required estimate.

\end{proof}
\section{A generalized Bando-Mabuchi theorem}
As pointed out to me by Robert Berman, Theorems 6.1 and 6.5  lead to  versions
of the Bando-Mabuchi theorem for 'twisted K\"ahler-Einstein
equations', \cite{Szekelyhidi}, \cite{2Berman}, and
\cite{2Donaldson}. Let $\theta$ be 
a positive $(1,1)$-current that can be written
$$
\theta=i\ddbar\psi 
$$
with $\psi$ a metric on a $\R$-line bundle $S$. The twisted
K\"ahler-Einstein equation is 
\be
\text{Ric}(\omega) =\omega +\theta,
\ee
for a K\"ahler metric $\omega$ in the class $c[-(K_X+ S)]$. Writing
$\omega=i\ddbar\phi$, where $\phi$ is a metric on the $\R$-line bundle 
 $F:=-(K_X+S)$, this is equivalent to
\be
(i\ddbar\phi)^n= e^{-(\phi+\psi)},
\ee
after adjusting constants. We will consider only bounded solutions to this equation. 

To be able to apply Theorems 6.1 and 6.4 we need to assume that $e^{-\psi}$ is
integrable. By this we mean that representatives with respect to a
local frame are integrable. When $\theta=[\Delta]$ is the current defined by a
divisor, it means that the divisor is klt.

Solutions $\phi$ of (7.2) are now critical points of the function
$$
D_\psi(\phi):=-\log\int e^{-(\phi+\psi)} -\E(\phi, \chi)
$$
where $\chi$ is an arbitary metric on $F$. Here  $\psi$ is fixed and we
let the variable $\phi$ range over bounded metrics with
$i\ddbar\phi\geq 0$. If $\phi_0$ and $\phi_1$ are two critical points,
it follows from the discussion in section 2  that
we can connect them with a bounded  geodesic $\phi_t$. Since $\E$ is
affine along 
the geodesic it follows that 
$$
t\rightarrow -\log\int e^{-(\phi_t+\psi)} 
$$
is affine along the geodesic and we can apply Theorem 6.1. 
\begin{thm} Assume that $-K_X$ is semipositive ( i e that it has a smooth metric of semipositive curvature) and that
  $H^{0,1}(X)=0$. 
Assume that  $i\ddbar\psi=\theta$, where
$e^{-\psi}$ is integrable and $\theta$ is a positive current. Let
$\phi_0$ and $\phi_1$ be two bounded solutions of equation (7.2)
with $i\ddbar\phi_j\geq 0$. Then there is a holomorphic vector field, $V$, with time one flow, $F$,
of $X$, homotopic to the identity, such that 
$$
F^*(\ddbar\phi_1)=\ddbar\phi_0.
$$
Moreover,  
$$
V\rfloor\theta=0
$$
and
$$
F^*(\theta)=\theta.
$$
\end{thm}

\begin{proof}The first part  follows immediately from Theorem 6.1.  Theorem 6.2 says that $V_t\rfloor\theta=0$.  This implies that the Lie derivative of $\theta$ along $V$ vanishes, which gives the last statement.  

\end{proof}

In the same way we get from Theorem 6.4
\begin{thm} Assume that $-K_X=-(K_X+S) +S $ where $-(K_X+S)$ is
  semipositive and $S$ is the $\R$-line bundle corresponding to a klt
  divisor $\Delta\geq 0$ with simple normal crossings. Assume also that 
  $H^{0,1}(X)=0$. 
  Let
$\phi_0$ and $\phi_1$ be two bounded solutions of equation (7.2) with
$\theta=[\Delta]$ and 
with $i\ddbar\phi_j\geq 0$. Then there is a holomorphic vector field,$V$, with time one flow, $F$, such that 
$$
F^*(\ddbar\phi_1)=\ddbar\phi_0.
$$
Moreover 
$$
V_t\rfloor \Delta=0
$$
so $V_t$ is tangential to $\Delta$, and
$$
F^*([\Delta])=[\Delta].
$$
\end{thm}
In some cases the conclusion of Theorems 7.1 and 7.2 actually imply that $V=0$, so that $F$ is the identity map and $\omega_0=\omega_1$. Probably the simplest case of this is the following (see the next section for variants on this). We assume that $X$ is Fano so that $-K_X>0$ and then let $(S)=-rK_X$, where $0<r<1$. Then we can rewrite equation (7.1) as 
$$
Ric(\omega)=(1-r)\omega +r\theta,
$$
where $\omega$ is a K\"ahler metric in $c_1[-K_X]$ and $\theta$ also lies in that class. We choose $\theta=[(1/\lambda)\Delta]$ where $\Delta$ is a smooth connected divisor of multiplicity one,  defined by a section $s$ of $-\lambda K_X$, $\lambda$ a positive integer. Then we can take $\psi$ in Theorem 7.1 as
$$
\psi=(r/\lambda)\log|s|^2.
$$
Clearly $e^{-\psi}$ is locally integrable and it follows from Theorem 7.1 that $V$ is tangential to the divisor $\Delta$ . But this implies that $V$ must be identically zero. This was proved by Berman for $\lambda=1$ and by Song and Wang for $\lambda\geq 1$; see \cite{2Berman} and \cite{Song-Wang}. (We will also give a different proof and partial extension for the case when $\lambda>1$ in the next section.) We summarize in a theorem.

\begin{thm} Let $[\Delta]$ be a smooth connected divisor of multiplicity one  on a Fano manifold $X$, defined by a section, $s$,  of $\lambda(- K_X)$, where $\lambda $ is a positive integer. Let $\omega_0$ and $\omega_1$ be two solutions in $c_1[-K_X]$ to the twisted K\"ahler-Einstein equation
$$
Ric(\omega)=(1-t)\omega +t/\lambda[\Delta],
$$
with $0<t<1$, of the form $\omega_j=i\ddbar\phi_j$, $\phi_j$ bounded.  Then $\omega_0=\omega_1$.
\end{thm}
Notice that the case $\lambda=1$ of this theorem is rather delicate. For $X$ equal to the Riemann sphere we can take the disconnected anticanonical divisor $\Delta=\{0, \infty\}$. Then clearly the conclusion of Theorem 7.3 fails as there are nontrivial automorphisms $z\mapsto az$ fixing $\Delta$. Thus the assumption of connectedness is necessary and, similarily, $\Delta=2\{0\}$ shows that we also need to assume multiplicity one. Note also, as pointed out by a referee, that $\Delta$ is automatically connected if $n>1$, as follows from the Lefschetz heyperplane theorem, \cite{Griffiths-Harris}.

 \section{ Complex
  gradients and uniqueness for twisted K\"ahler-Einstein equations} 

The main point of the proofs in the previous sections was that we found a holomorphic vector field, $V$, on $X$  satisfying
$$
V\rfloor\ddbar\phi=\dbar\dot{\phi},
$$
so that $V$ was sort of a 'complex gradient' in a rather non regular situation. This vector field also satisfied 
$$
V\rfloor\theta=0
$$
where $\theta$ is the twisting term in the twisted K\"ahler-Einstein equations. We will now discuss when this last condition forces $V$ to be zero. 
Mainly to illustrate the idea we start with a situation when the metric is smooth, but not necessarily positively curved. 
\begin{prop} Let $L$ be a holomorphic line bundle over the compact
  K\"ahler manifold $X$, and let $\psi$ be a  smooth metric on $L$, not
  necessarily with positive 
  curvature. Assume that
$$
H^{(0,1)}(X, K_X+L)=0
$$
and
$$
H^{0}(X, K_X+L)\neq 0.
$$
Assume also that $V$ is a holomorphic vector field on $X$ such that
$$
V\rfloor \ddbar\psi=0.
$$
Then $V=0$.

\end{prop}
\begin{proof} We follow the arguments in section 3. Let $u$ be a
  global holomorphic section of $K_X+L$, and put
$$
v:=V\rfloor u.
$$
Then $v$ is a holomorphic, $L$-valued, $(n-1,0)$-form and
$$
\ddbar\psi\wedge v=-(V\rfloor\ddbar\psi)\wedge u=0.
$$
Hence
$$
\dbar\partial^\psi v=-\partial^\psi\dbar v=0.
$$
This implies that
$$
\int \partial^\psi v\wedge\overline{\partial^\psi v}e^{-\psi}=
\pm\int \dbar\partial^\psi v\wedge\overline{ v}e^{-\psi}=0.
$$
Hence $\partial^\psi v=0$. Moreover, our assumption that $H^1$ vanishes implies that the $\dbar$-closed form $v\wedge\omega=\dbar w$ for some $(n,0)$-form $w$. 
Therefore 
$$
\int v\wedge\bar v\wedge\omega e^{-\psi}=\int v\wedge\overline{\dbar w} e^{-\psi}=0.
$$
Thus $v$ and $V$ must be zero which completes the proof. 
\end{proof}

\bigskip

We shall next see that the same conclusion holds if we only assume that our metric is such that $e^{-\psi}$ is locally integrable if we assume that the curvature current is semipositive. 
 \begin{prop} Let $L$ be a holomorphic line bundle over the compact
  K\"ahler manifold $X$, and let $\psi$ be a   metric on $L$ such that $i\ddbar\psi\geq 0$ and $e^{-\psi}$ is locally integrable. Assume that
$$
H^{(0,1)}(X, K_X+L)=0
$$
and
$$
H^{0}(X, K_X+L)\neq 0.
$$
Assume also that $V$ is a holomorphic vector field on $X$ such that
$$
V\rfloor \ddbar\psi=0.
$$
Then $V=0$.
\end{prop}
For the proof we need a technical lemma which is a little bit more delicate than it seems at first glance. Recall that when $\psi$ is smooth we have defined the expression $\partial^\psi v$ as
$$
\partial^\psi v=e^{\psi}\partial e^{-\psi} v=\partial v-\partial\psi\wedge v.
$$
More exactly, this means that these relations hold in any local trivialization.

\bigskip

As it stands, the first of these formulas does not make sense if $\psi$ is allowed to be singular. Therefore we {\it define}, for $v$ smooth
and $\psi$ singular
\be
\partial^\psi v=\partial v-\partial\psi\wedge v.
\ee
The next lemma says that the  formula used above continues to hold in the singular case if we assume that $e^{-\psi}$ is locally integrable and the weight is plurisubharmonic. (But not in general, see below!) 

\begin{lma} Let $\psi$ be a plurisubharmonic function in an open set in $\C^n$ such that $e^{-\psi}$is locally integrable. Let $v$ be a smooth differential form such that $\partial^\psi v$, as defined in 8.1,  is also smooth. Then 
$$
e^{-\psi}\partial^\psi v=\partial e^{-\psi} v
$$
in the sense of currents.
\end{lma}
\begin{proof}
What we need to check is that 
$$
v\wedge \partial e^{-\psi}=-(v\wedge \partial \psi) e^{-\psi}
$$
if $v$ and $-v\wedge \partial \psi$ are smooth so that the expressions above are well defined. Since the statement is purely local we can take a sequence of smooth plurisubharmonic functions $\psi_\nu=\psi\ast \chi_\nu$, where $\chi_\nu$ is a sequence of radial approximations of the identity, that decrease  to $\psi$. The left hand side is then the limit in the sense of currents of
$$
-v\wedge \partial\psi_\nu e^{-\psi_\nu}
$$
and the right hand side is the limit of 
$$
-v\wedge \partial\psi e^{-\psi_\nu}.
$$
We have to prove that these two limits are equal.
\begin{lma} If $\psi$ is plurisubharmonic, then 
$\partial\psi$ lies in $L^p$ locally for any $p<2$. With $\psi_\nu$ as above $\partial\psi_\nu$ tends to $\partial\psi$ in $L^p_{loc}$ for any $p<2$.
\end{lma} 
\begin{proof}For any compact $K$, there is an $\epsilon$ such that $e^{-\epsilon\psi}$ is integrable over $K$, see \cite{Hormander}. Moreover, since 
$$
i\ddbar e^{\epsilon\psi}\geq \epsilon^2 e^{\epsilon\psi} i\partial\psi\wedge\dbar\psi,
$$
we see that $|\partial\psi|^2 e^{\epsilon\psi}$ is locally integrable for any $\epsilon>0$. Therefore H\"older's inequality implies that  $|\partial\psi|^p$ is locally integrable for any $p<2$. Hence 
$$
\partial\psi_\nu=\partial\psi\ast \chi_\nu
$$
tend to $\partial\psi$ in $L^p_{loc}$.
\end{proof}

Let us now first assume that $e^{-\psi}$ is not only integrable locally, but lies in $L^q$ locally for some $q>2$. Then the conclusion of the lemma follows from the convergence of $\partial\psi_v$ in $L^p$ and of $e^{-\psi_\nu}$ in $L^q$. 

 This means that under the assumption in the lemma we have proved that
$$
v\wedge \partial e^{-t\psi}=-tv\wedge \partial \psi e^{-t\psi},
$$
if $t$ is between zero and one half. But both sides are real analytic functions of $t$ with values in the space of currents, for $t<1$. Therefore the same formula holds for any $t$ less than one and we only need to take limits (still in the space of currents) as $t$ tends to one.
\end{proof}

\noindent{\bf Example:} Let $\psi=\log |z|^2$ in $\C$ and let $v=z$. The $v$ is smooth and $\partial^\psi v=0$. On the other hand 
$$
\partial e^{-\psi} v=\partial \frac{1}{\bar z}=\delta_0 dz\neq 0.
$$
This shows that the assumption of local integrability of $e^{-\psi}$ is essential. Otherwise the two sides do not need to be equal even if they are well defined.
\qed

\bigskip

\noindent {\bf Proof of Proposition 8.2:} Given the lemma, this proceeds just like the proof of Proposition 8.1. Take $u$ a holomorphic section of $K_X+L$, and let $v=V\rfloor u$. As before $v$ is holomorphic and $\partial^\psi v$ is also holomorphic. In particular, booth forms are smooth so we can apply the lemma. First
$$
c_n\int \partial^\psi v\wedge\overline{\partial^\psi v}e^{-\psi}=\int \partial^\psi v\wedge\dbar (e^{-\psi}\bar v)= \int \dbar\partial^\psi v\wedge\bar v e^{-\psi}=0.
$$
Hence $\partial^\psi v=0$.  Then, invoking the K\"ahler form $\omega$, if $v\wedge\omega =\dbar w$,
$$
c_{n-1}\int v\wedge \bar v\wedge \omega e^{-\psi}=
c_{n-1}\int v\wedge \overline{\dbar w} e^{-\psi}=
\pm\int \partial^{\psi} v\wedge \bar w e^{-\psi}=0.
$$
Hence $v$ and therefore $V$ vanish.
\qed

To have an example of the situation in Proposition 8.2, look at a smooth divisor, $\Delta$  defined by a section $s$ of a multiple $\lambda L$ of $L$. Let $ \psi=(1/\lambda)\log |s|^2$. Then $\psi$ satisfies the assumption of Proposition 8.2 if $\lambda >1$. This means that any holomorphic vector field that is tangential to $\Delta$ (in particular, vanishing on $\Delta$) must vanish, cf \cite{Song-Wang}. As reflected by the example above, this is not true if $\lambda =1$. For an example of this take a field $V$ on the Riemann sphere that vanishes at zero and infinity. Concretely,  $z\partial/\partial z$ on $\C$ extends to such a field. Here $L=-K_{\Po^1}$ . The cohomological assumptions of Proposition 8.2 are satisfied, but the conclusion fails if $\psi$ is the metric on $-K_{\Po^1}=O(2)$, that extends $\log |z|^2$ on $\C$. However, Song-Wang in the reference above and also Berman, \cite{2Berman} have proved that the conclusion does hold on a Fano manifold for $L=-K_X$ for an anticanonical divisor, provided the divisor is smooth, connected of multiplicity one. This does not seem to follow from our propositions.

\subsection{Meromorphic vector fields}
In the next section we will need an extension of the results of the previous section when $V$ is only known to be meromorphic. In this case we cannot expect anything as precise as Proposition 8.2, even if the poles of $V$ lie outside of the support of $\ddbar\psi$. Let for example $\psi$  be a metric on the anticanonical bundle of the Riemann sphere (i e on $O(2)$) that equals
$$
\psi'=\frac{2}{N}\sum_1^N\log |z-a_i|^2
$$
on $\C$. Since $\psi'$ grows like $2\log|z|^2$ at infinity, infinity is outside the support of $\ddbar\psi$. Let
$$
V=\Pi_1^n (z-a_i)\frac{\partial}{\partial z}
$$
on $\C$; it extends to a meromorphic field with pole at infinity. Thus the conclusion of Proposition 8.2 fails even though $e^{-p\psi}$ is integrable for $p<N/2$. On the other hand, we shall now see that if $L$ is ample, and $e^{-k\psi}$ is integrable for {\it all} $k$, the proposition 8.2 holds even for a meromorphic field.

 \begin{prop} Let $L$ be an ample  holomorphic line bundle over the compact
  K\"ahler manifold $X$, and let $\psi$ be a  metric on $L$ such that $i\ddbar\psi\geq 0$ and $e^{-k\psi}$ is locally integrable for all $k>0$.
Let  $V$ be a meromorphic vector field on $X$ such that
\be
V\rfloor \ddbar\psi=0
\ee
outside the poles of $V$. 
Then $V=0$.
\end{prop}
\begin{proof}
Since $V$ is meromorphic, there is a holomorphic section $s$ of some holomorphic line bundle $(S)$ such that $sV$ is holomorphic. Taking $k$ sufficiently large we can, since $L$ is ample, find a nontrivial holomorphic section $u'$ of $K_X+kL-(S)$. Let $u=su'$. Then $u$ is a holomorphic section of 
$K_X+kL$, and $v:=V\rfloor u$ is also holomorphic.  As before, the condition 8.3 implies that  $\ddbar\psi\wedge v$ is zero outside of the polar divisor. Therefore it vanishes everywhere since $\ddbar\psi$ cannot charge any divisor if $e^{-k\psi}$ is integrable for all $k$. We can then repeat the proof of Proposition 8.2 word for word, if we replace $L$ by $kL$.
\end{proof}

\section{A variant of Theorem 1.2 for other line bundles than $-K_X$}

In this section we consider a general semipositive line bundle $L$ over $X$, and the space of holomorphic sections $H^0(X,K_X+L)$. First we assume that this space is nontrivial, but later we will even assume that $K_X+L$ is base point free, i e that the elements of  $H^0(X,K_X+L)$ have no common zeroes. Let $\Omega$ be an open set in $\C$; it will later be the strip $\{t; 0<\Re t<1\}$, and let $\phi_t$ for $t$ in $\Omega$ be a subgeodesic (see section 2.2) in the space of metrics on $L$. As  explained in the beginning of section 3 we then get a trivial vector bundle 
$E$ over $\Omega$ with fiber  $H^0(X,K_X+L)$ with norm
$$
\|u\|_t^2=\int_X |u|^2 e^{-\phi_t}.
$$
In case $\phi_t$ is smooth, the curvature of this metric is given by Theorem 3.1. From this formula we see that the curvature $\Theta$ is nonnegative, and if for some $u$ in  $H^0(X,K_X+L)$, $\Theta u=0$ at a point $t$, then $v_t$ is holomorphic. We can then follow the same route as before and define a vector field $V$ by
$$
V\rfloor u=v_t.
$$
In section 3 we looked at the case when $L=-K_X$, in which case $K_X+L$ is trivial, so a holomorphic section has no zeros, and it follows that $V$ is a holomorphic field. For general $L$, $u$ will have zeros, so $V$ is a priori only meromorphic. If $\phi_t$ is smooth and has strictly positive curvature this is not a serious problem since
$$
V\rfloor \ddbar\phi_t=\dbar\dot{\phi_t}
$$
so $V$ is smooth and therefore must after all be holomorphic. Therefore the arguments of section 3 lead to the conclusion that if $E$ does not have strictly positive curvature then the change of metric must be given by the flow of a holomorphic vector field; see \cite{2Berndtsson}. 

If e g the subgeodesic is only bounded this argument does not work. Nevertheless we can by adopting the methods of section 4  get the same conclusion if we assume that the curvature is not only degenerate, but vanishes identically. Since the curvature is always nonnegative the assumtion amounts to saying that the trace of the curvature vanishes. This is the same as saying that the determinant of $E$ has zero curvature. Yet another way of saying the same thing is in terms of the function
$$
\mathcal{L}(\phi_t):=\log \text{Vol}(B_t),
$$
where $B_t$ is the unit ball in $E_t$, and the volume $\text{Vol}$ is computed with respect to some fixed Lebesgue measure on  $H^0(X,K_X+L)$. 
It follows from Theorem 3.1 that $\mathcal{L}(\phi_t)$ is concave along a subgeodesic and the curvature of $E$ is zero if and only if  $\mathcal{L}(\phi_t)$ is affine.

\begin{thm} Assume that
$K_X+L$
is base point free and that
$$
H^1(X,K_X+L)=0.
$$
Assume that $\phi_t$ is a bounded subgeodesic in the space of metrics on $L$ which is independent of the imaginary part of $t$. Then, if $\mathcal{L}(\phi_t)$ is affine, there is a holomorphic vector field $V$ on $X$ with flow $F_t$ such that 
$$
F_t^*(\ddbar\phi_t)=\ddbar\phi_0.
$$
\end{thm}
\begin{proof}
As explained above the assumtion that $\mathcal{L}$ be affine means that the curvature of $E$ vanishes identically. Following the arguments of section 4, we get for each $t$ and each $u$ in  $H^0(X,K_X+L)$ a meromorphic vectorfield $V_{t,u}$ satisfying
$$
V_{t,u}\rfloor\ddbar\phi_t=\dbar\dot{\phi_t}.
$$
By Proposition 8.4 this means that $V_{t,u}=V_{t, u'}$, so all the fields for different choice of sections are the same. Since the poles of $V_{t,u}$ are contained in the zero set of $u$, and since we have assumed that our bundle is base point free it follows that there are no poles. The proof is then concluded in the same way as in section 4.
\end{proof}

\section{K\"ahler-Ricci solitons}
Let $X$ be a Fano manifold. A  K\"ahler form, $\omega$,  on $X$ in $c[-K_X]$ is said to be a K\"ahler-Ricci soliton if it satisfies the equation
\be
Ric(\omega)=\omega +L_V\omega
\ee
for some holomorphic vector field $V$ on $X$. Here $L_V$ is the Lie derivative of $\omega$ along $V$ which is also equal to $L_{\Re V}+i L_{\Im V}$. Taking real and imaginary parts we see that 
$$
 L_{\Im V}\omega=0
$$
so $\omega $ is invariant under the flow of the imaginary part of $V$. Tian and Zhu, see \cite{Tian-Zhu} have proved a generalization of the Bando-Mabuchi uniqueness theorem for solitons, saying that two solutions to (10.1) are related via the flow of a holomorphic vector field. In this section we shall show that this theorem also follows from Theorem 1.2. In a later paper, \cite{2Tian-Zhu}, Tian and Zhu also proved a generalization of this result for two solitons that a priori are associated to different vector fields. The proof there builds on their earlier result. For completeness we also sketch the very beautiful reduction of Tian-Zhu of the general problem to the problem for one fixed field, although we only have minor simplifications for that part of the argument. (See also  \cite{He-man}.)
\subsection{ Solitons associated to one fixed field}
As in the proof of Tian-Zhu our proof uses a generalization of the energy functional $\E$, that was introduced by Zhu in \cite{Zhu}, which we shall now describe. 

Let first 
$$
\H=\{\phi;\text{ metric on }\, -K_X, i\ddbar\phi\geq 0\}.
$$
In the sequel we write $\omega^\phi$ for $i\ddbar\phi$ if $\phi$ lies in $\H$.
If $V$ is any holomorphic vector field on $X$ and $\phi$ lies in $\H$, we will define a function $h^\phi$ by
$$
V\rfloor\omega^\phi=i\dbar h^\phi.
$$
This definition is meaningful since the left hand side is clearly a $\dbar$-closed $(0,1)$-form, and therefore $\dbar$-exact on a Fano manifold. 
Of course, $h^\phi$ is only determined up to a constant, and we will need to choose this constant in a coherent way. Here is one way to do this.
\begin{df} Let $\phi$ be a smooth metric on $-K_X$, and let as in section 2, $e^{-\phi}$ be the corresponding volume form on $X$. Then we define, for $V$ an arbitrary holomorphic vector field on $X$, the function $h^\phi_V$ by
\be
L_V(e^{-\phi})=-h^\phi_V e^{-\phi},
\ee
where $L_V$ is the Lie derivative. 
\end{df}

\begin{prop} The function $h^\phi_V$ defined in 10.2 has the following properties:

\medskip

1. $V\rfloor i\ddbar\phi=i\dbar h^\phi$,

\medskip

2. $\int h^\phi_V e^{-\phi}=0$,

\medskip

3. $h^\phi_{V+V'}=h^\phi_V +h^\phi_{V'}$,

\medskip

and if $\chi$ is a function

\medskip

4. $h^{\phi+\chi}_V =h^\phi_V+ V(\chi)$.

\medskip

Moreover, $h^\phi_V$ is real valued if and only if $\omega^\phi=i\ddbar\phi$ is invariant under the flow of $\Im V$.
\end{prop}
\begin{proof} Property 2 follows since the integral of a  Lie derivative of a volume form always vanishes, 3 is direct from the formula for the  Lie derivative and 4 follows since $h^\phi_V$ is a logarithmic derivative. 

To check 1, choose local coordinates $z^j$ and a local representative of $\phi$ so that
$$
e^{-\phi}=c_n e^{-\phi_j} dz^j\wedge d\bar z^j.
$$
Let $F_t(z)$ be the flow of $V$. Then for small $t$
$$
F_t^*(e^{-\phi})=c_n  e^{-\phi_j\circ F_t} |J(t,z)|^2 dz^j\wedge d\bar z^j,
$$
where the Jacobian $J$ is holomorphic in $t$ and $z$ jointly. 
Then 
$$
h^\phi_V=\frac{\partial}{\partial t}( \phi_j\circ F_t -\log |J(t,z)|^2)|_{t=0}=V(\phi_j)+ R.
$$
Here $R$ is holomorphic in $z$, so
$$
i\dbar h^\phi_V=i\dbar V(\phi_j)=V\rfloor i\ddbar\phi.
$$
For the final claim, note that 1 implies that
$$
2\Im V\rfloor i\ddbar\phi= d\Re h^\phi +d^c \Im h^\phi.
$$
Hence
$L_{\Im V}\omega^\phi=0$ if and only if the imaginary part of $h^\phi_V$ is constant. By 2, this constant must be zero. 
\end{proof}
\begin{preremark} Since 1 and 2 of Proposition 10.1 determine $h^\phi_V$ uniquely,  we could also have defined $h^\phi_V$ by properties 1 and 2. This is the route taken in \cite{Tian-Zhu} and \cite{2Tian-Zhu}. We have chosen to start instead from 10.2 since it seems to simplify the argument somewhat and also gives 3 and 4 for free.
\end{preremark}

\bigskip

Next we let
$$
\H_{V}=\{\phi\in \H_X; L_{\Im V}i\ddbar\phi=0\}.
$$
We can now define Zhu's energy functional by, if $\phi_t$ is a smooth curve in $\H_V$, 
\be
\frac{d}{dt}\E_{Z,a}(\phi_t):=-\int \dof (i\ddbar\phi_t)^n e^{a h^{\phi_t}}
\ee
where $a$ is some real constant. In the sequel we will suppress the subscript $a$, and in the end we will choose $a=1$, but it seems useful to include an arbitrary $a$ in the discussion anyway.  Of course we need to prove that $\E_Z$ is well defined, cf \cite{Zhu}. This and other basic properties of $\E_Z$ are summarized in the next propositions.
\begin{prop}
Let $\phi_{s,t}$ be a smoothly parametrized family of metrics in $\H_V$. Then
$$
\frac{d}{ds}\int \frac{d\phi_{t,s}}{dt} (i\ddbar\phi_{t,s})^n e^{a h^{\phi_{t,s}}}=
\int(\frac{d^2\phi_{t,s}}{dtds}-\langle\dbar\dot{\phi_t},\dbar\dot{\phi_s}\rangle)(i\ddbar\phi_{t,s})^n e^{a h^{\phi_{t,s}}}.
$$
\end{prop}
The bracket $\langle\dbar\dot{\phi_t},\dbar\dot{\phi_s}\rangle$ stands here for the real scalar product of the two forms with respect to the metric $i\ddbar\phi_{t,s}$, i e the real part of the complex scalar product. 
Since the right hand side is therefore symmetric in $t$ and $s$ it follows that $\E_Z$ is well defined and putting $t=s$ that
\begin{prop}
$$
\frac{d^2}{dt^2}\E_Z=-\int(\ddot{\phi_{tt}}-|\dbar\dof|^2)(i\ddbar\phi_t)^n e^{ah^{\phi_t}}.
$$
\end{prop}

Moreover, taking $\phi_{t,s}=\phi_s+t$ we see that
$$
\int(i\ddbar\phi)^n e^{ah^{\phi}}
$$
is constant on $\H_V$. This is of course where the specific choice of $h^\phi$ is important. Proposition 10.2
 is essentially contained in Zhu's paper but formulated differently there and we will give a proof in an appendix. 
\begin{prop} $\E_Z$ is affine along the  $C^{1,1}$-geodesic connecting two smooth metrics in $\H_V$.
\end{prop}
\begin{proof} It is well known that
$$
(\ddot{\phi_{tt}}-|\dbar\dof|^2)(i\ddbar\phi_t)^n\wedge idt\wedge d\bar t=(1/n)(i\ddbar\phi)^{n+1}),
$$
where $(i\ddbar\phi)^{n+1}$ is the Monge-Ampere measure of $\phi$ with respect to all variables on $\Omega\times X$. The formula in Proposition 10.3 can therefore be interpreted as saying that
$$
\ddbar \E_Z (\phi_t) =(-1/n)p_*((i\ddbar\phi)^{n+1}e^{a h^{\phi_t}}).
$$
This was proved assuming that $\phi$ is smooth so we need to regularize $\phi$ if it is only of class $C^{1,1}$. Moreover, we need regularize so that we stay in the space $\H_V$. This is actually achieved by Chen's proof of the $C^{1,1}$-regularity of geodesics. There the geodesic is obtained as the limit of smooth  $\epsilon$-geodesics that are solutions  of a strictly elliptic equation. These $\epsilon$-geodesics are invariant under $\Im V$ if the boundary values are. 

It is well known that the Monge-Ampere measure converges weakly under decreasing limits of bounded plurisubharmonic functions. Moreover it is clear from our formula for $h^{\phi_t}$ that $h^{\phi_t}$  converges uniformly under limits in $C^1$. Therefore the formula holds also if $\phi$ is only in $C^{1,1}$. Since the Monge-Ampere measure of a (generalized) geodesic vanishes the claim follows.
\end{proof}
\begin{preremark} It is also true that $\E_Z$ is affine along any $C^1$-geodesic in $\H_V$. This can be proved by approximating a $C^1$ -geodesic in $\H_V$ by smooth curves in $\H_V$, but we omit the details.
\end{preremark}

\bigskip

Let us now see how the uniqueness theorem of Tian-Zhu follows from Theorem 1.2 and Proposition 10.2. Let $\omega_0=i\ddbar\phi_0$ and $\omega_1=i\ddbar\phi_1$ be two (smooth) solutions to equation 10.1. As noted above $\phi_j$ lie in $\H_V$. To avoid technical complications we will resort to Chen's theorem and connect $\phi_0$ and $\phi_1$ with a $C^{1,1}$-geodesic, $\phi_t$. By uniqueness of geodesics it follows that for all $t$ $\phi_t$ lies in $\H_V$. Since
$$
L_Vi\ddbar\phi=d (V\rfloor i\ddbar\phi)= d i\dbar h^\phi=i\ddbar h^\phi,
$$
we can rewrite equation 10.1 as 
\be
(i\ddbar\phi)^ne^{h^\phi}=C e^{-\phi}
\ee
for some constant $C$. Choose $C_0$ so that 
$$
C_0\int(i\ddbar\phi)^n e^{h^{\phi}}=1.
$$
Then (10.4) implies that
$$
C_0 (i\ddbar\phi)^n e^{h^{\phi}}=\frac{e^{-\phi}}{\int e^{-\phi}}.
$$
Define
$$
\F_Z(\phi)=\log\int e^{-\phi} -C_0\E_Z(\phi).
$$
Any solution of equation 10.4 in $\H_V$ is a critical point of $\F_Z$. Since $\E_Z$ is affine along our geodesic we see by Theorem 1.1 that $\F_Z(\phi_t)$ is convex in $t$. If $\phi_0$ and $\phi_1$ are critical points $\F_Z(\phi_t)$ and hence
$$
\log\int e^{-\phi_t}
$$
are also affine in $t$. Theorem 1.2 then implies that there is a holomorphic vector field on $X$ with time one flow $F$ such that $F^*(\omega_1)=\omega_0$. 

\subsection{ Solitons associated to different fields}

In this section we sketch the arguments of Tian-Zhu from \cite{2Tian-Zhu} to prove that two solutions to the equations 
\be
Ric(\omega_0)=\omega_0 +L_V\omega_0
\ee
and
\be
Ric(\omega_1)=\omega_1 +L_W\omega_1,
\ee
where $V$ and $W$ are two holomorphic vector fields on $X$, are also connected via an automorphism in $Aut_0(X)$, the connected component of the identity in the automorphism group of $X$. As we have seen $\Im V$ generates a flow of isometries for $\omega_0$. This flow is contained in a maximal compact subgroup $K_0$ of $Aut_0(X)$. In the same way, the flow of $\Im W$ is contained in another maximal compact subgroup, $K_1$. By a fundamental theorem of Iwasawa, \cite{Iwasawa}, $K_0$ and $K_1$ are conjugate by an automorphism $g$ in $Aut_0$.
This means that after a preliminary automorphism applied to one of the equations, we may assume that $K_0=K_1=:K$. Tian-Zhu then show that this implies that $V=W$. 

To explain how this is done we go back to the construction of the function $h^\phi$ in the previous subsection, with basic properties described in Proposition 10.1 .
Notice that property 4 of Proposition 10.1 means that if $\phi=\phi_t$ depends on a real parameter, then $(d/dt)h^{\phi_t}_V=V(\dot{\phi}_t)$. Following \cite{2Tian-Zhu} we now define a functional
$$
f(\phi,V):=\int e^{h^\phi_V}(i\ddbar\phi)^n,
$$
for $\phi$ in $\H$ and $V$ any holomorphic vector field on $X$.
\begin{prop}
$f(\phi,V)$ does not depend on $\phi$.
\end{prop}
\begin{proof} Take $\phi=\phi_t$ and differentiate with respect to $t$:
$$
\frac{d}{dt}f(\phi_t,V)=\int V(\dot{\phi}_t) e^{h^\phi_V}(i\ddbar\phi)^n+ n\int
 e^{h^\phi_V}i\ddbar\dot{\phi}_t\wedge(i\ddbar\phi)^{n-1}
$$
But, since contraction with $V$ is an antiderivation, 
$$
V(\dot{\phi}_t)(i\ddbar\phi)^n=ni\partial\dot{\phi}_t\wedge\dbar h^\phi_V\wedge
(i\ddbar\phi)^{n-1}.
$$
Inserting this and integrating by parts we see that the derivative of $f$ with respect to $t$ vanishes, so $f$ does not depend on $\phi$. 
\end{proof} 
In the sequel we write $f(\phi,V)=f(V)$.
\begin{prop} Suppose the holomorphic vector field $V$ admits a K\"ahler-Ricci soliton, i e that there is a solution $\omega$ to the equation
$$
Ric(\omega)=\omega +L_V(\omega).
$$
Then $V$ is a critical point of $f$. 
\end{prop}
\begin{proof}
By property 3 of Proposition 10.4 the derivative of $f$ at the point $V$ in the direction $U$ is
$$
\int h^\phi_U e^{h^\phi_V}(i\ddbar\phi)^n.
$$
Here we can choose any $\phi$ in $\H$ by Proposition 10.5. If we take $i\ddbar\phi$ to be a $V$-soliton, then
 $ e^{h^\phi_V}(i\ddbar\phi)^n= C e^{-\phi}$ by 10.4.  Hence the derivative is zero for any $U$ by property 2 of proposition 10.4
\end{proof}
\bigskip

Recall that $K$ is a compact subgroup of $Aut_0(X)$ which contains the flows of both $\Im V$ and $\Im W$. Let $f_K$ be the restriction of $f$  to the space of all vector fields that have this property, i e whose imaginary part lie in the Lie algebra of $K$. Choose $\omega=\omega^{\phi}$ to be $K$-invariant; this can be achieved by taking averages with respect to the Haar measure of $K$. Write
$$
f_K(V)=\int e^{h^\phi_V}\omega^n.
$$
By Proposition 10.4 all $h^\phi_V$ are real valued if $V$ is such a field and $i\ddbar\phi$ is $K$-invariant. Since moreover $h^\phi_V$ is linear in $V$ by property 3 of Proposition 10.1, this formula shows that $f_K$  is strictly convex. 
Therefore $f_K$ can have at most one critical point. It follows immediately that there is at most one vector field $V$ with $\Im V$ in the Lie algebra of $K$ that admits a soliton. In other words, $V$ and $W$ from the beginning of this subsection must be equal (after the preliminary reduction).
By subsection 10.1 we then arrive at the following theorem.
\begin{thm}(Tian-Zhu, \cite{Tian-Zhu}, \cite{2Tian-Zhu}) Let $X$ be a Fano manifold and let $\omega_0$ and $\omega_1$ be solutions of 10.4 and 10.5. Then there is an automorphism $g$ in $Aut_0(X)$ such that $g^*(\omega_1)=\omega_0$.
\end{thm}
\section{Appendix}
Here we will prove Proposition 10.1. We suppress the dependence of $\phi$ on $t$ and $s$ in the formulas and use subscripts only to denote differentiation with respect to these variables. 
$$
\frac{d}{ds}\int \dof (i\ddbar\phi)^n e^{a h^{\phi}}=\int \ddot{\phi}_{t,s} (i\ddbar\phi)^n e^{a h^{\phi}} +
$$
$$
+n\int\dof(i\ddbar\dot{\phi}_s)\wedge(i\ddbar\phi)^{n-1}  e^{a h^{\phi}}
+a\int \dof V(\dot{\phi}_s)\wedge(i\ddbar\phi)^{n}  e^{a h^{\phi}}=:I+II+III.
$$
Integrating by parts we get
$$
II=- n\int i\partial\dot{\phi_t}\wedge\dbar\dot{\phi_s}\wedge(i\ddbar\phi)^{n-1} e^{ah^\phi}-an\int i\dof\partial h^{\phi}\wedge\dbar\dot{\phi_s}\wedge(i\ddbar\phi)^{n-1} e^{ah^\phi}.
$$
Recall that $i\dbar h^\phi=V\rfloor i\ddbar\phi$ so that we have  $-i\partial h^\phi=\bar V\rfloor i\ddbar\phi.$ Since contraction with a vector field is an antiderivation we get
$$
0=\bar V\rfloor(\dbar\dot{\phi_s}\wedge(i\ddbar\phi)^n)=\overline{V(\dot{\phi_s})}
(i\ddbar\phi)^n+n\dbar\dot{\phi_s}\wedge i\partial h^\phi\wedge (i\ddbar\phi)^{n-1}.
$$
Inserting this above we see that
$$
II=- n\int i\partial\dot{\phi_t}\wedge\dbar\dot{\phi_s}\wedge(i\ddbar\phi)^{n-1} e^{ah^\phi}-a\int\dof \overline{V(\dot{\phi_s})}(i\ddbar\phi)^ne^{ah^\phi}.
$$
Hence 
$$
\frac{d}{ds}\int \dof (i\ddbar\phi)^n e^{a h^{\phi}}=
\int \ddot{\phi}_{t,s} (i\ddbar\phi)^n e^{a h^{\phi}}- n\int i\partial\dot{\phi_t}\wedge\dbar\dot{\phi_s}\wedge(i\ddbar\phi)^{n-1} e^{ah^\phi} +
$$
$$
2ia\int \dof \Im V(\dot{\phi}_s)\wedge(i\ddbar\phi)^{n}  e^{a h^{\phi}}.
$$
But the left hand side of this equality is  real so  the last term must be zero (which is also clear since $\phi$ is invariant under the flow of $\Im V$).   We are then left with the formula in Proposition 10.1

\def\listing#1#2#3{{\sc #1}:\ {\it #2}, \ #3.}

\end{document}